\documentclass[11pt]{article}
\usepackage[utf8]{inputenc}
\usepackage{color}
\usepackage{xcolor}
\usepackage{ulem}
\usepackage{lipsum}
\usepackage{authblk}

\usepackage{amsmath,amsfonts,euscript,amssymb,amsthm,a4,times}

\allowdisplaybreaks

\def\XXint#1#2#3{{\setbox0=\hbox{$#1{#2#3}{\int}$}
\vcenter{\hbox{$#2#3$}}\kern-.5\wd0}}

\newtheorem{theorem}{Theorem}[section]

\theoremstyle{definition}
\newtheorem{definition}[theorem]{Definition}

\makeatletter \catcode`@=11
\newbox\tr@tto
\setbox\tr@tto=\hbox{{\count0=0\dimen0=-,9pt\dimen1=1,1pt\loop\ifnum
    \count0<11 \advance \count0 by1 \vrule width.51pt height\dimen1
    depth\dimen0\kern-0.17pt\advance\dimen0 by-0.05pt\advance\dimen1
    by0.1pt\repeat \loop\ifnum\count0<21\advance \count0 by1 \vrule
    width.6pt height\dimen1 depth\dimen0\kern-0.2pt \advance\dimen0
    by-0.1pt\advance\dimen1 by 0.05pt\repeat}}
\def\medint{\displaystyle\copy\tr@tto\kern-10.4pt\int}
\numberwithin{equation}{section}

\def\proofof#1{\begin{proof}[Proof of #1]}

\def\dx{{\mathrm d}x}

\def\loc{\rm loc}

\def\proofof#1{\begin{proof}[Proof of #1]}

\def\dx{{\mathrm d}x}

\def\loc{\rm loc}

\catcode`@=11
\newbox\tr@tto
\setbox\tr@tto=\hbox{{\count0=0\dimen0=-,9pt\dimen1=1,1pt\loop\ifnum\count0<11 \advance \count0 by1 \vrule width.51pt height\dimen1
     depth\dimen0\kern-0.17pt\advance\dimen0 by-0.05pt\advance\dimen1
     by0.1pt\repeat \loop\ifnum\count0<21\advance \count0 by1 \vrule
     width.6pt height\dimen1 depth\dimen0\kern-0.2pt \advance\dimen0
     by-0.1pt\advance\dimen1 by 0.05pt\repeat}}
\def\medint{\displaystyle\copy\tr@tto\kern-10.4pt\int}
\catcode`@=12

\numberwithin{equation}{section}

\theoremstyle{plain} 
\newtheorem{thm}{Theorem}[section] 
 
\newtheorem{lem}[thm]{Lemma} 
\newtheorem{prop}[thm]{Proposition} 
 
\newtheorem{dfn}[thm]{Definition}

\usepackage{xcolor}                    
\definecolor{custom-blue}{RGB}{0,99,166} 
\usepackage{hyperref}
\hypersetup{colorlinks=true, allcolors=custom-blue}

\numberwithin{equation}{section}

\date{\today}

\begin{document}
\title{ Regularity for Minimizers of non Autonomous Singular Functionals with Anisotropic Growth }
\author{\sc{Albert Clop},  \sc{Antonia Passarelli di Napoli} $\&$ \sc{Stefania Russo}}

\newcommand{\Addresses}{{
  \bigskip
  \footnotesize
\noindent{\sc Albert Clop} , Departament de Matematiques i Informatica  \\ Universitat de Barcelona, Carrer d'Aribau, 2, Eixample, 08011 Barcelona, Espanya\par\nopagebreak
\noindent\texttt{albert.clop@ub.edu}

\medskip

\noindent {\sc Antonia Passarelli di Napoli}, Dipartimento di Matematica e Applicazioni  "R.
Caccioppoli"\\ Universit\`{a} di Napoli ``Federico
II", via Cintia - 80126 Napoli, Italy\par\nopagebreak
\noindent\texttt{antpassa@unina.it}

\medskip

\noindent {\sc Stefania Russo}, Dipartimento di Matematica e Applicazioni "R.
Caccioppoli" \\ Universit\`{a} di Napoli ``Federico
II", via Cintia - 80126 Napoli, Italy\par\nopagebreak
\noindent\texttt{stefania.russo3@unina.it}
}}

\maketitle

\bigskip

\begin{abstract}
We establish the higher differentiability  of the local minimizers to a class of non autonomous convex integral functionals satisfying anisotropic subquadratic growth conditions, that include, as a particular case, those with orthotropic structure. 
The result is obtained under a gap bound on the exponents \(p_i\), that guarantees the local boundedness of the minimizers and under a suitable Sobolev assumption on the  map that measures the oscillation of the energy density with respect to the $x$ variable, that is independent on the dimension.
\end{abstract}
\maketitle
\bigskip

\section{Introduction}
In this paper we investigate the higher differentiability properties of the local minimizers of convex integral functionals of the form
\begin{equation}\label{functional}
\mathcal{F}(u,\Omega):=\int_\Omega f(x,Du)\,\dx,
\end{equation}
where $\Omega\subset\mathbb{R}^n$ is a bounded open set, $n\ge 2$, $u:\Omega\to\mathbb{R}$, and the integrand $f:\Omega\times\mathbb{R}^n\to[0,+\infty)$ 
is assumed to be strictly convex and of class $C^1$ with respect to the variable $\xi$ and satisfies the following anisotropic growth conditions 
\begin{equation}\label{crescitapi}
  c_1 \sum_{i=1}^{n} (\mu^2+|\xi_i|^2)^{\frac{p_i}{2}}
\leq
f(x,\xi)
\leq
c_2\left(1+\sum_{i=1}^{n} |\xi_i|^{p_i}\right)
\end{equation}
 for some exponents $p_i$, with 
 \begin{equation}\label{pi}
1<p_1\le\dots\le p_n\le 2.
\end{equation}
More precisely, we assume that there exists a function $\tilde f:\Omega\times\mathbb{R}^n\to[0,+\infty)$ such that
\begin{equation}\label{A0}f(x,\xi)=\tilde f(x, |\xi_1|,\dots,|\xi_n|)\tag{A0}\end{equation}  and that there exists $\bar{\mu}>1$ such that
\begin{equation}\label{A1}
f(x,\lambda \xi) \leq \lambda^{\bar{\mu}} f(x,\xi),\qquad\text{for every   }\lambda>1.
\tag{A1}
\end{equation}
Moreover, there exist positive constants $\ell,L$ such that
\begin{equation}\label{A2}
\langle D_\xi f(x,\xi)-D_\xi f(x,\eta),\xi-\eta\rangle
\ge \ell\sum_{i=1}^n(\mu^2+|\xi_i|^2+|\eta_i|^2)^{\frac{p_i-2}{2}}|\xi_i-\eta_i|^2,
\tag{A2}
\end{equation}
\begin{equation}\label{A3}
|D_\xi f(x,\xi)-D_\xi f(x,\eta)|
\le L\sum_{i=1}^n(\mu^2+|\xi_i|^2+|\eta_i|^2)^{\frac{p_i-2}{2}}|\xi_i-\eta_i|,
\tag{A3}
\end{equation}
for a.e. $x\in\Omega$ and all $\xi,\eta\in\mathbb{R}^n$. For the sake of simplicity,  we assume
\[
D_\xi f(x,0)=0
\qquad\text{for a.e. }x\in\Omega,
\]
that, together with (A3), implies
\begin{equation}\label{A3'}
|D_\xi f(x,\xi)|
\le L\sum_{i=1}^n(\mu^2+|\xi_i|^2)^{\frac{p_i-1}{2}}.
\tag{A3'}
\end{equation}
However, one can easily check that our results still hold under the weaker assumption $D_\xi f(x,0)=\mathfrak{h}(x)\in L^m_{\mathrm{loc}}(\Omega)$, 
 for a suitable integrability exponent \(m>1\). 
\\
A prototypical example of the class of energies considered in this paper is given by the orthotropic integrand
\begin{equation}\label{esplicita}
f(x,\xi)=\sum_{i=1}^n a_i(x)\bigl(\mu^2+|\xi_i|^2\bigr)^{\frac{p_i}{2}},
\end{equation}
where $\mu\in[0,1]$ is fixed, and the coefficients
$a_i:\Omega\to\mathbb{R}$ are  measurable and essentially bounded functions fulfilling the uniform ellipticity condition
\begin{equation}\label{ellipticity}
0<\lambda\le a_i(x)\le\Lambda
\qquad\text{for a.e. }x\in\Omega,\quad i=1,\dots,n.
\end{equation}
The regularity of the local minimizers of anisotropic functionals has been widely investigated in both the homogeneous (\cite{6, 5, LemmaBrasco, Carozza, 17, feopass,  russo})  and non-homogeneous cases (\cite{BraTau,uLimitato,russogri}) as well as that of the solutions of degenerate anisotropic parabolic equations (see  \cite{Pasq,PasqCiani}).\\
Actually,   functionals as in \eqref{functional} under the growth assumption  \eqref{crescitapi},  fit into the wider class  of nonstandard growth problems. Indeed, setting
\[
p:=\min_i p_i,
\qquad
q:=\max_i p_i,
\]
the integrand satisfies the so-called  $ (p,q)$–growth condition, that can be expressed as follows
$$|\xi|^p -c_0\leq f(x,\xi) \leq c_1(1 + |\xi|^2)^{\frac{q}{2}}. $$ 
Since the pioneering works of Marcellini   (\cite{MarcelliniEs1, MarcelliniEs2}), it has been known that the ratio $q/p$ plays a decisive role in the regularity theory: if the ratio  is too large, local minimizers may fail to be bounded, as proven by the counterexamples in \cite{Giaquinta, MarcelliniEs3}. \\
In the non-autonomous setting the situation is even more delicate due to the possible occurrence of the Lavrentiev phenomenon (\cite{BDS, FDFL, Koch}), and admissible bounds on the ratio $q/p$  could also depend also on the degree of regularity  of the coefficients.

Over the years, the literature on minimizers of functionals with non-standard growth conditions has grown considerably, making it impossible to provide an exhaustive list of references  (see, for example   \cite{CKP1, CKP2}, \cite{12}---\cite{Ele3}, \cite{16, M2022, MascoloPdN}) and  the recent surveys (\cite{surveyMar1,surveyMing}).\\
Although our functional shares features with the $(p,q)$–growth framework, its orthotropic structure requires a refined analysis, since each component of the gradient has its own growth exponent.
As pointed out in \cite{LemmaBrasco}, the functional combines two different features: $(p,q)$-growth and orthotropic structure. This interaction gives rise to a particularly challenging class of variational problems. Actually, the classical $(p,q)$-growth setting does not provide sufficient information to develop a regularity theory that fully exploits the orthotropic structure. In particular, taking into account the orthotropic structure allows us to obtain stronger results than those available by considering only the $(p,q)$-growth behavior.
Indeed, in the $(p,q)$ framework  typically 
the regularity results are formulated in terms of the whole gradient and only involve the lowest exponent. In contrast, in the orthotropic setting one proves that each component of the gradient 
gains   a regularity  related to the corresponding growth exponent, that is obviously  better.  

\noindent Given the anisotropic structure of the integrand, it is natural to work in the Sobolev space
\[
W^{1,\mathbf{p}}_{\mathrm{loc}}(\Omega)
=
\left\{u\in W^{1,1}_{\mathrm{loc}}(\Omega): u_{x_i}\in L^{p_i}_{\mathrm{loc}}(\Omega),\ i=1,\dots,n\right\},
\]
where $\mathbf{p}=(p_1,\dots,p_n)$. Let us recall the following

\begin{definition}
A function $u\in W^{1,\mathbf{p}}_{\mathrm{loc}}(\Omega)$ is a local minimizer of \eqref{functional} if for every open set $\widetilde\Omega\Subset\Omega$ one has
\[
\mathcal{F}(u;\widetilde\Omega)\le \mathcal{F}(\varphi;\widetilde\Omega)
\]
for all $\varphi\in u+W^{1,\mathbf{p}}_0(\widetilde\Omega)$.
\end{definition}
The regularity theory for local minimizers of autonomous orthotropic functionals, i.e. with energy densities of the form \eqref{esplicita} with $a_i\equiv1$, has been recently object of study, 
in the subquadratic  (see \cite{Brascop12})  and superquadratic (see \cite{BraTau}) regimes.   In particular, if
\[
u\in W^{1,\mathbf{p}}_{\mathrm{loc}}(\Omega)\cap L^\infty_{\mathrm{loc}}(\Omega)
\]
is a local minimizer of the corresponding autonomous functional,
and one further assumes that $$1\le p_1\le\dots\le p_n\le 2$$
then in \cite{Brascop12} it has been proven  that
\begin{equation}\label{diffproperties}
|u_{x_i}|^{\frac{p_i-2}{2}}u_{x_i}\in W^{1,2}_{\mathrm{loc}}(\Omega),
\qquad i=1,\dots,n.
\end{equation}
Our main result extends  the above conclusion to minimizers of anisotropic integral functionals  \eqref{functional},  not necessarily autonomous, as already done in \cite{russo} in the superquadratic growth case.
\\
It is well known that, for non-autonomous functionals, the higher differentiability of local minimizers cannot in general be expected unless some regularity with respect to the $x$–variable is imposed on the integrand. In both the standard and nonstandard growth frameworks, Sobolev regularity of the coefficients provides a natural condition ensuring higher differentiability of integer order.
The available results make it clear that the regularity of the coefficients needs to be related to the dimension $n$, in the sense that the partial map $x \to D_{\xi}f(x, \xi)$ is usually assumed to   belong to a Sobolev space  strictly contained in $W^{1,n}$ even under  standard growth conditions (see for example \cite{CGPdN, CGHPdN20, CGGP, CMMP, Ming2, Ele1, Ele2, Ele4, M2022, 40}.
When the minimizers are assumed to be essentially  bounded, however, the situation changes, and sufficient assumptions on the Sobolev regularity of the coefficients can be imposed independently of the dimension (see \cite{Adi, Ele3, 25, Hasto}).

\noindent Accordingly, concerning the dependence on the $x$-variable, we assume that for every $ i=1,\dots,n$
there exists a function $g \in L^r_{\textrm{loc}}(\Omega) $ for some $ r>1$  such that
\begin{equation}\label{A4}
|D_\xi f(x,\xi)-D_\xi f(y,\xi)|
\le |x-y|\,(g(x)+g(y))
\sum_{i=1}^n(\mu^2+|\xi_i|^2)^{\frac{p_i-1}{2}},
\tag{A4}
\end{equation}
for a.e. $x,y\in\Omega$ and all $\xi\in\mathbb{R}^n$.\\
Under this assumption, also in the orthotropic subquadratic growth case, we are able to prove a higher differentiability result for the local minimizers of \eqref{functional}, that as usual,  need to be expressed through the  auxiliary function $V_{p}(\xi)$  defined by
\begin{equation}\label{DefVp}
    V_{p}(\xi):= V_{p,0}(\xi); \qquad 
    V_{p,\mu}(\xi):= \left( \mu^2 + |\xi|^2 \right)^{\frac{p-2}{4}} \xi,
    \end{equation}
for all $\xi\in \mathbb{R}^{n}$ and with $ \mu  \in [0,1]$. 
\noindent 
More precisely, we are going to prove the following
\begin{theorem}\label{thmPrincipale}
Let $u\in W^{1,\mathbf{p}}_{\mathrm{loc}}(\Omega)$ be a local minimizer of \eqref{functional} under assumptions \eqref{A0}--\eqref{A4} and \eqref{crescitapi}--\eqref{pi}. Assume moreover that
$$p_n \leq \bar{p}^*$$
where $ \bar{p}^*$ is the anisotropic Sobolev exponent of $ \mathbf{p},$ defined at \eqref{definizionePi},
and that
\[
g\in L^r_{\mathrm{loc}}(\Omega)
\qquad\text{with}\qquad
r>p_n+2.
\]
Then
\[
V_{p_i}(u_{x_i})\in W^{1,2}_{\mathrm{loc}}(\Omega),
\qquad i=1,\dots,n,\qquad \text{and}\qquad  u \in  W^{1,\mathbf{p}+2}_{\mathrm{loc}} (\Omega). 
\]
Moreover , for every pair of concentric balls $B_{R/4}\subset B_R\Subset\Omega$, the estimates
\[
\sum_{i=1}^n\int_{B_{R/4}}
(\mu^2+|u_{x_i}|^2)^{\frac{p_i-2}{2}}
|u_{x_ix_j}|^2\,\dx
\le
c\Bigl(
\sum_{i=1}^n\|u_{x_i}\|_{L^{p_i}(B_R)}
+\|g\|_{L^r(B_R)}
\Bigr)^\sigma
\]
and 
\[
\sum_{i=1}^n\int_{B_{R/4}}
(\mu^2+|u_{x_i}|^2)^{\frac{p_i+2}{2}}\,\dx
\le
c\Bigl(
\sum_{i=1}^n\|u_{x_i}\|_{L^{p_i}(B_R)}
+\|g\|_{L^r(B_R)}
\Bigr)^\sigma
\]
hold, where $c>0$ and $\sigma>0$ depend only on the structural data.
\end{theorem}

The proof relies on an approximation procedure combined with uniform a priori estimates. We first establish the higher differentiability of the minimizers of suitably regularized functionals , that is with Lipschitz continuous coefficients, by means of the difference quotient techniques adapted to the anisotropic structure. This is the first main step in our proof, since it provides the necessary regularity of the approximating minimizers. We then derive uniform estimates independent of the regularization parameter; this step represents the core of our argument. It is achieved by using  the well-known difference quotient method and  this is the point where the subquadratic growth comes into play and requires specific techniques. The final step consists in constructing a family of approximating problems whose minimizers are regular enough to satisfy the assumptions of the a priori estimate  and then proving that the estimate is preserved in  the limiting argument.\\
The assumption \( p_n \le \bar{p}^* \) is required only to address the case of bounded minimizers. In its absence---namely, when minimizers may be unbounded, as illustrated by the counterexample in~\cite{MarcelliniEs1, MarcelliniEs2}---our results can still be established under a stronger condition on \( g \), specifically by assuming \( g \in L^r \) with \( r > n \).\\
The paper is organized as follows. After collecting preliminary material and notation, we introduce the approximating functionals and prove auxiliary results needed for the a priori estimates. We then establish uniform bounds for the regularized problems and finally pass to the limit, thereby completing the proof of Theorem \ref{thmPrincipale}.

\section{Preliminaries}
In this section we introduce some notations and collect several results that we shall use to establish our main result.
We   denote by $c$ or $C$ a general constant that may vary on different occasions, even within the same line of estimates. Relevant dependencies will be emphasized if necessary.

In what follows, $B(x,r)=B_{r}(x)= \{ y \in \mathbb{R}^{n} : |y-x | < r  \}$ will denote the ball centered at $x$ of radius $r$. We will omit the indication of the center $x$ when no confusion arises.

We recall here the well-known Sobolev–Troisi inequalities \cite{troisi} (see also \cite{feopass} for a weighted version). 
\begin{lem}\label{LemmaTroisi}
    Let $E \subset \mathbb{R}^n$ be a bounded open set and consider $u \in W_0^{1, \mathbf{p}}(E), \, p_i \geq 1$ for all $i=1, ..., n,$ and set
    \begin{align}\label{definizionePi}
        \frac{1}{\overline{p}} = \frac{1}{n} \sum_{i=1}^{n} \frac{1}{p_i}, \qquad \overline{p}* = \frac{n \overline{p}}{n- \overline{p}}.
    \end{align}
    Then
    \begin{itemize}
        \item[1)]  If $\overline{p} < n$,  there exists a positive constant $\gamma$ depending only on the  parameters $\{ n, p_1, \dots, p_n \}$, such that
    \begin{align*}
        \lvert \lvert u \rvert \rvert_{\overline{p}*}\leq \gamma \sum_{i=1}^{n} \lvert \lvert u_{x_i} \rvert \rvert_{p_i}.
    \end{align*}
    \item[2)] If $\overline{p} = n$, for every $1 \leq r < \infty$  there exists a positive constant $\gamma_1$ depending only on  $\{ n, r, p_1, \dots, p_n \}$, such that
    \begin{align*}
        \lvert \lvert u \rvert \rvert_{r} \leq \gamma_1  \left[ \sum_{i=1}^{n} \lvert \lvert u_{x_i} \rvert \rvert_{p_i} + \lvert \lvert u \rvert \rvert_{1} \right].
    \end{align*}
  \item[3)]     If $\overline{p} > n$,  there exists a positive constant $\gamma_2$ depending only on  $\{ n, p_1, \dots, p_n \}$, such that
    \begin{align*}
        \lvert \lvert u \rvert \rvert_{\infty} \leq \gamma_2  \left[ \sum_{i=1}^{n} \lvert \lvert u_{x_i} \rvert \rvert_{p_i} + \lvert \lvert u \rvert \rvert_{1} \right].
    \end{align*}
    \end{itemize}
\end{lem} 
The next   lemma (see \cite[Lemma $6.1$]{giusti}) is a well-known  tool in the regularity theory.
\begin{lem}\label{lm3} 
Let $R>0$ and let $Z(t)  :  [\rho,R] \rightarrow \mathbb{R}$ be a bounded nonnegative function. Assume that for $\rho \leq t  < s \leq R$ it holds
$$Z(t) \leq \theta Z(s) +A(s-t)^{- \alpha} + B(s-t)^{- \beta}+ C$$
with constants $\theta \in (0,1)$, $A$, $B$, $C \geq 0$ and $\alpha> \beta >0$. Then there exists a constant $c=c(\alpha, \theta, \beta)$ such that
$$Z(\rho) \leq c \biggl( A(R-\rho )^{- \alpha} + B(R-\rho )^{- \beta}+ C\biggr).$$
\end{lem}
The following interpolation type inequality will become significant later. The proof can be found in  \cite[Proposition 4.3]{LemmaBrasco} in a slightly different form; other interpolation inequalities of this type have been used in the regularity theory for a priori bounded minimizers and can be found in \cite{CKP1,25}. We report it here for the sake of completeness.

\begin{prop}\label{LemmaBrasco}
 Let $ \mathbf{q}=(q_1, \dots, q_n)$ with $q_1>1$ and let $u \in W^{1,\mathbf{q}}_{\mathrm{loc}}(\Omega)$. Assume that there exists   $i \in \{1,2,\dots,n \}$ such that
    $$|u_{x_i x_i}|^2 \left( \mu^2 + |u_{x_i}|^2 \right)^\frac{q_i-2}{2} \in L^{1}_{\mathrm{loc}}(\Omega).$$
   If $u \in L^{\infty}_{\mathrm{loc}}(\Omega),$  then $u_{x_i} \in L_{\mathrm{\mathrm{loc}}}^{q_i+2}(\Omega)$  and there exist constants $C=C(n, q_i )>0$ and $\gamma= \gamma(q_i)>0$ such that the following estimate
    \begin{align}\label{LemmaBr}
        \int_{B_\rho} \left( \mu^2 + |u_{x_i}|^2 \right)^\frac{q_i+2}{2}  \dx &\leq C ||u||^2_{\infty}  \, \int_{B_R} \left( \mu^2 + |u_{x_i}|^2 \right)^\frac{q_i-2}{2} |u_{x_ix_i}|^2  \dx \notag \\
        & \qquad +\frac{C ||u||^2_{\infty}}{(R- \rho)^\gamma}  \int_{B_R} \left( \mu^2 + |u_{x_i}|^2 \right)^\frac{q_i}{2} \dx 
    \end{align}
   holds for every pair of concentric balls $B_{\rho} \subset B_{R}\Subset  \Omega$.
\end{prop}
\begin{proof}
   We fix  a ball $B_R \Subset \Omega$,  consider radii $\rho <s<t<R$ and  a positive cut-off function $\eta \in C_0^{\infty}(B_R)$ , with $\eta=1$ on $B_s$, $0 \leq \eta \leq 1$ and $ |D\eta| \leq \frac{c}{t-s}.$\\
 First, we observe that
    \begin{eqnarray}\label{InterpBrasco}
        &&\int_{\Omega} \eta^2 \big( \mu^2 + |u_{x_i}|^2 \big)^\frac{q_i+2}{2}\dx \cr\cr 
         &\leq &\int_{\Omega} \!\!\eta^2 \mu^2\! \left( \mu^2 + |u_{x_i}|^2 \right)^\frac{q_i}{2} \dx +  \underbrace{\int_{\Omega} \eta^2  \left( \mu^2 + |u_{x_i}|^2 \right)^\frac{q_i}{2} |u_{x_i}|^2 \dx}_{\textbf{\textbf{I}}},
    \end{eqnarray}
 and we estimate $\textbf{I}$  integrating by parts as follows
    \begin{align*}
     \textbf{I} &=  \int_{\Omega}  \eta^2 \left( \mu^2 + |u_{x_i}|^2 \right)^\frac{q_i}{2}u_{x_i}\,u_{x_i} \dx =  -\int_{\Omega} u D_{x_i} \left( \eta^2 \left( \mu^2 + |u_{x_i}|^2 \right)^\frac{q_i}{2}u_{x_i} \right) \dx \notag \\
        &= -\int_{\Omega} u \Bigg[ \eta^2 u_{x_ix_i}\left( \mu^2 + |u_{x_i}|^2 \right)^\frac{q_i}{2} 
        + q_i \eta^2 u_{x_i}u_{x_i x_i}\left( \mu^2 + |u_{x_i}|^2 \right)^\frac{q_i-2}{2}u_{x_i} \notag \\
        & \qquad\qquad + 2 \eta \eta_{x_i}u_{x_i}\left( \mu^2 + |u_{x_i}|^2 \right)^\frac{q_i}{2} \Bigg] \dx .
    \end{align*}
    Since $u \in L^{\infty}_{\loc}(\Omega)$ and by virtue of the properties of $\eta$,
    we have
     \begin{align}\label{A+Bintegrals}
     |\textbf{I} |\leq & \lvert \lvert u \rvert \rvert_{L^{\infty}} \Bigg[  \int_{B_t}   |u_{x_i x_i}| \,\left( \mu^2 + |u_{x_i}|^2 \right)^\frac{q_i}{2} \dx 
     +  q_i\int_{B_t}  |u_{x_i x_i}| |u_{x_i}|^2\left( \mu^2 + |u_{x_i}|^2 \right)^\frac{q_i -2 }{2}   \dx  \notag \\
        & \qquad +\frac{c}{t-s} \int_{B_t}  \left( \mu^2 + |u_{x_i}|^2 \right)^\frac{q_i+1}{2} \dx\Bigg]\notag \\
        &=:c \lvert \lvert u \rvert \rvert_{L^\infty} \big[ \textbf{I}_1+ \textbf{I}_2 + \textbf{I}_3\big].
    \end{align}
      Young's inequality yields
    \begin{align}\label{Aintegral}
       \textbf{I}_1
        & \leq  \varepsilon \int_{B_t} \left( \mu^2 + |u_{x_i}|^2 \right)^\frac{q_i +2}{2} \,  \dx + c_\varepsilon \int_{B_t} |u_{x_i x_i}|^2 \, \left( \mu^2 + |u_{x_i}|^2 \right)^\frac{q_i-2}{2}  \dx ,
    \end{align}
    \begin{align}\label{Bintegral}
      \textbf{I}_2
        & \leq  \int_{B_t}  \left( \mu^2 + |u_{x_i}|^2 \right)^{\frac{q_i-2}{2}} \, \left (\mu^2+|u_{x_i}|^2 \right)\,  |u_{x_i x_i}| \, \dx \notag \\
        & \leq  \varepsilon \int_{B_t}  \left( \mu^2 + |u_{x_i}|^2 \right)^\frac{q_i +2}{2} \,  \dx + c_\varepsilon \int_{B_t} |u_{x_i x_i}|^2 \, \left( \mu^2 + |u_{x_i}|^2 \right)^\frac{q_i-2}{2}  \dx 
    \end{align} 
      and
    \begin{align}\label{Cintegral}
      \textbf{I}_3
         & \leq  \varepsilon \int_{B_t}  \left( \mu^2 + |u_{x_i}|^2 \right)^{\frac{q_i+2}{2}} \,  \dx + \frac{c_\varepsilon}{(t-s)^2}\int_{B_t}\left( \mu^2 + |u_{x_i}|^2 \right)^\frac{q_i}{2} \, \dx,
    \end{align}
      where $\varepsilon>0$ will be chosen later. 
  Combining  \eqref{Aintegral}, \eqref{Bintegral} and \eqref{Cintegral} with \eqref{A+Bintegrals}, we deduce that
   \begin{align}\label{Iintbrasco}
       |\textbf{I}| & \leq 3\varepsilon \int_{B_t}  \left( \mu^2 + |u_{x_i}|^2 \right)^\frac{q_i +2}{2} \,  \dx + c_\varepsilon \int_{B_t} |u_{x_i x_i}|^2 \, \left( \mu^2 + |u_{x_i}|^2 \right)^\frac{q_i-2}{2}  \dx \notag \\
       & \qquad + \frac{c_\varepsilon}{(t-s)^2}\int_{B_t}\left( \mu^2 + |u_{x_i}|^2 \right)^\frac{q_i}{2} \, \dx,
   \end{align}
   and so, inserting \eqref{Iintbrasco} in \eqref{InterpBrasco}, we obtain
    \begin{align}
      \int_{B_s} &  \left( \mu^2 + |u_{x_i}|^2 \right)^\frac{q_i +2}{2}\dx  \leq c \lvert \lvert u\rvert \rvert_{L^\infty} \Bigg[   3\varepsilon \int_{B_t}  \left( \mu^2 + |u_{x_i}|^2 \right)^\frac{q_i +2}{2} \,  \dx \notag \\
      & \qquad +  \left(\frac{c_\varepsilon }{(t-s)^2} + \mu^2 \right) \int_{B_t}  \,  \left( \mu^2 + |u_{x_i}|^2 \right)^\frac{q_i}{2} \dx 
      + c_\varepsilon \int_{B_t} |u_{x_i x_i}|^2 \, \left( \mu^2 + |u_{x_i}|^2 \right)^\frac{q_i-2}{2}  \dx  \Bigg] .
    \end{align}
    Note that we may suppose that $\lvert \lvert u \rvert \rvert_\infty\not=0$, otherwise there is nothing to prove. Hence, we may choose   $\varepsilon = \frac{1}{6c \lvert \lvert u \rvert \rvert_\infty} $ and, applying Lemma \ref{lm3}, we conclude that
\begin{align*}
 &\int_{B_\rho}    \left( \mu^2 + |u_{x_i}|^2 \right)^\frac{q_i +2}{2}\dx\\ &\leq\ \frac{c}{(R- \rho)^2} \int_{B_R}  \left( \mu^2 + |u_{x_i}|^2 \right)^\frac{q_i}{2}\dx  + c \int_{B_R} |u_{x_i x_i}|^2 \, \left( \mu^2 + |u_{x_i}|^2 \right)^\frac{q_i-2}{2} \dx,
 \end{align*}
i.e. the thesis. 
\end{proof}  
We shall use the following local boundedness result for minimizers of the  anisotropic functional defined at \eqref{functional}, whose proof can be found in \cite[Theorem 1.1]{uLimitato}.
{ \begin{thm}\label{ulimitatoCupini}
    Let us consider the  functional 
    \begin{equation}\label{functional3}
    \mathcal{I}(v)=\int_\Omega g(x,Dv)\dx, 
    \end{equation}
with the integrand $g(x,\xi)$  satisfying  \eqref{crescitapi}, \eqref{A0}, \eqref{A1} and \eqref{A3'} with exponents $1<p_1\le\dots\le p_n$ such that $$p_n < \bar{p}^*,$$ where $\bar{p}^*$ has been defined at \eqref{definizionePi},
then every local minimizer $u$ of \eqref{functional3} is locally bounded.
\end{thm}}
As usual when dealing with nonlinear problem, the extra differentiability of the local minimizers is expressed through  the function $V_{p, \mu}(\xi)$,
 defined at \eqref{DefVp}, that takes into account  the precise growth of  the integrand.

For further needs, we recall the following 
\begin{lem} \label{VpAndrea}
Let \(1< p \le 2 \). There exists positive constants \( c_1 ,c_2,c_3 \)   such that
\begin{equation*} \label{eq:2.3}
    c_1^{-1}\left( \mu^2 + |\xi|^2 + |\eta|^2 \right)^{\frac{ p-2}{4}} \leq 
   \frac{ \left| V_{p, \mu}(\xi) - V_{p, \mu}(\eta) \right| }{ |\xi - \eta|}
      \leq c_1 \left( \mu^2 + |\xi|^2 + |\eta|^2 \right)^{\frac{ p-2}{4}},
\end{equation*}
and
 \begin{eqnarray*}
 \Big\langle \left( \mu^2 + |\xi|^2 \right)^{\frac{ p-2}{2}}\xi- \left( \mu^2 + |\eta|^2 \right)^{\frac{ p-2}{2}}\eta, \xi-\eta\Big\rangle  &\geq&   c_2\, (\mu^2+|\xi|^{2}+|\eta|^{2})^{\frac{p-2}{2}}|\xi-\eta|^2 \cr\cr&\ge &   c_3 \left| V_{p, \mu}(\xi) - V_{p, \mu}(\eta) \right|^2 \end{eqnarray*}  
for every \( \xi, \eta \in \mathbb{R}^{n} \).
\end{lem}

For the proof we refer to \cite[Lemma 2.1]{Fusco}.

\subsection{Difference quotient}
\label{secquo}
In this section we recall some properties of the finite difference quotient operator that will be needed in the sequel. 

\begin{dfn}
Let $F$ be a function defined in an open set $\Omega \subset \mathbb{R}^n$, let $h$ be a real number,  the finite difference operator $\tau_{s,h}F(x)$ is defined as follows
$$ 
\tau_{s,h}F(x) :=F(x+he_s)-F(x) ,$$
where $e_s$ denotes the direction of the $x_s$ axis.
\end{dfn}
The function $\tau_{s,h}F$ is defined in the set
$$\Omega_{|h|}: = \{ x \in \Omega : \mathrm{dist}(x,\partial \Omega)> |h|  \}= \{  x \in \Omega : x+he_s \in \Omega \}.$$
We start with the description of some elementary properties that can be found, for example, in \cite{giusti}. When no confusion  arises, we shall omit the index $s$, and we shall write simply $\tau_{h}$ instead of $ \tau_{s, h}$
\begin{prop}\label{rapportoincrementale}
Let $F \in W^{1,p}(\Omega)$, with $p \geq1$, and let $G:\Omega \rightarrow \mathbb{R}$ be a measurable function.
Then
\\(i) $\tau_{h}F \in W^{1,p}(\Omega_{|h|})$ and 
$$D_{i}(\tau_{h}F)=\tau_{h}(D_{i}F).$$
(ii) If at least one of the functions $F$ or $G$ has support contained in $\Omega_{|h|}$, then
$$\displaystyle\int_{\Omega}F \tau_h G   \dx = \displaystyle\int_{\Omega} G \tau_{-h}F \dx.$$
(iii) We have $$\tau_{h} (FG)(x)= F(x+h)\tau_{h} G(x)+G(x) \tau_{h} F(x).$$
\end{prop}
The next result about the finite difference operator is a kind of integral version of Lagrange Theorem.
\begin{lem}\label{ldiff}
If $0<\rho<R,$ $|h|<\frac{R-\rho}{2},$ $1<p<+\infty$ and $F, \, D_s F\in L^{p}(B_{R})$, then
\begin{center}
$\displaystyle\int_{B_{\rho}} |\tau_{s, h}F(x)|^{p} \dx \leq c(n,p)|h|^{p} \displaystyle\int_{B_{R}} |D_sF(x)|^{p} \dx$.
\end{center}
Moreover,
\begin{center}
$\displaystyle\int_{B_{\rho}} |F(x+h)|^{p} d x \leq  \displaystyle\int_{B_{R}} |F(x)|^{p}d x$.
\end{center}
\end{lem}
\noindent We conclude by recalling a result  that will be needed in the sequel.

\begin{lem}\label{Lemmahzero}
Let $ 0<\rho<R,$ $1<p<+\infty$ and $F\in L^{p}(B_{R})$. If there exists a positive constant $C$ such that
\begin{align}
\displaystyle\int_{B_{\rho}} |\tau_{s, h}F(x)|^{p} \dx \leq C|h|^{p}, 
\end{align}
for every $h$ such that $|h|<\frac{R-\rho}{2},$ then the distributional derivative $D_s F$ belongs to $L^p(B_\rho)$ and the following estimate
\begin{align}
\displaystyle\int_{B_{\rho}} |DF(x)|^{p} \dx \leq C 
\end{align}
holds.
\end{lem}

\section{A preliminary regularity result}\label{PreR}
In this section, we establish a regularity result for a class of problems exhibiting higher regularity assumptions compared to those associated with \eqref{functional}. This result will play a crucial role in the subsequent approximation procedure, providing the necessary analytical framework to justify the limiting arguments and ensuring the regularity of the minimizers of the approximating sequence.
\\
More precisely, let  $F(x,\xi): \Omega\times\mathbb{R}^n\to[0,+\infty)$  be a function satisfying \eqref{crescitapi}, \eqref{A0}, \eqref{A1}, \eqref{A2}, \eqref{A3} and \eqref{A3'} with the same constants and the same exponents $p_i$. Assume moreover that there exists a constant $K>0$  such that
\begin{equation}\label{A'4}
    |D_\xi  F(x,\xi)-D_\xi F(y, \xi)| \leq K |x-y|\, \, \sum_{i=1}^{n} \left( \mu^2+ |\xi_i|^2 \right)^\frac{p_i-1}{2}, 
    \tag{A4'}
\end{equation}
\noindent for a.e.\ $x,y \in \Omega $ and every $\xi \in \mathbb{R}^{ n}$.
\\
For $\varepsilon \geq 0$, let us consider 
\begin{equation}\label{Eqconepsilon}
 F_\varepsilon(x, \xi)=F(x,\xi)+ \varepsilon(1+|\xi|^\frac{p_n}{2})^2
\end{equation}
and observe that,  by virtue of \eqref{A2}, \eqref{A3} and Lemma \ref{VpAndrea}, the following conditions  
\begin{eqnarray}\label{A2primo}
    \langle D_\xi F_\varepsilon(x, \xi)-D_\xi F_\varepsilon(x, \eta),\xi-\eta\rangle  &\geq&  \ell \sum_{i=1}^{n}(\mu^2+|\xi_i|^2+|\eta_i |^2)^\frac{p_i -2}{2}|\xi_i - \eta_i |^2\cr\cr
    &&+c(p_n)\varepsilon(1+|\xi|^2)^\frac{p_n -2}{2}|\xi - \eta |^2
\end{eqnarray}
\begin{eqnarray}\label{A3primo}
     |D_\xi F_\varepsilon(x, \xi)-D_\xi F_\varepsilon(x, \eta)| &\leq&  L \sum_{i=1}^{n}{( \mu^2 + |\xi_i |^2+|\eta_i |^2)^\frac{p_i -2}{2}}|\xi_i - \eta_i | \cr\cr
     &&+C(p_n)\varepsilon(1+|\xi|^2)^\frac{p_n -2}{2}|\xi - \eta |,
\end{eqnarray}
 hold for a. e. $x \in \Omega$ and all $\xi, \eta \in \mathbb{R}^{ n}$. Estimate \eqref{A2primo}  implies that
 \begin{equation}\label{A4primo}
    \langle D_\xi F_\varepsilon(x, \xi)-D_\xi F_\varepsilon(x, \eta),\xi-\eta\rangle  \geq  c(p_n)\varepsilon(1+|\xi|^2)^\frac{p_n -2}{2}|\xi - \eta |^2
\end{equation}
 and, from the definition of $F_\varepsilon$ and \eqref{A3'}, also
 \begin{eqnarray}\label{A5primo}
     |D_\xi F_\varepsilon(x, \xi)|\le C(n,p_n, L)(1+|\xi|^2)^\frac{p_n-1 }{2}.
\end{eqnarray}
Moreover, we explicitly remark that \eqref{A'4} holds true also for $F_\varepsilon$ with the same constant.

For later use, we recall a Lipschitz regularity result for the minimizers of the functional \begin{equation}\label{functional2}\mathcal{F}_\varepsilon (x, Dv) =\int_\Omega F_\varepsilon(x, Dv),\end{equation} that can be deduced by  \cite[Theorem $2.2$]{GriMasPass} in the case $p=q=p_n$. 
\\
It is worth observing that the integrands \(F_\varepsilon\) are energy densities with standard \(p_n\)-growth, but they possess the special structure given by \eqref{A0}. Therefore, the regularity of the minimizers of the functionals \(\mathcal{F}_\varepsilon\) can be deduced from results concerning functionals with anisotropic structure and standard growth. Indeed, the term $ \varepsilon(1+|\xi|^\frac{p_n}{2})^2$, ensuring the standard growth, is only meant to allow the application of the aforementioned results without the gap bound otherwise required to obtain them.

\begin{thm}\label{thm2Mascolo}
    Let $u_\varepsilon \in W^{1,p_n}(\Omega)$ be a local minimizer of the functional $ \mathcal{F}_\varepsilon (u,\Omega)$ defined at \eqref{functional2} under the assumptions \eqref{A2}---\eqref{A'4} and \eqref{A2primo}---\eqref{A5primo} .
Then $u_\varepsilon \in W^{1,\infty}_{\mathrm{loc}}(\Omega)$ and the following estimate 
    \begin{equation}\label{lipestimate}
      \Vert D u_\varepsilon \Vert_{L^\infty(B_\rho)} \le C \left[1+ \int_{B_R} f_\varepsilon(x,D u_\varepsilon)\dx  \right]^\frac{\sigma}{p_1},  
    \end{equation}
     holds for every pair of concentric balls $B_\rho \subset B_R \Subset \Omega$, with a  positive constant $C$ depending on $n,p_n,\rho,R,\varepsilon$ and $\sigma=\sigma(p_n,n)>0$.
\end{thm}
\noindent We shall need the higher differentiability of the minimizers $u_\varepsilon$ of $\mathcal{F}_\varepsilon$ that, to the best of our knowledge, is not available in the literature in the subquadratic growth cases and that is contained in the following 
\begin{lem}\label{LemmaInduzione}
  Let $F_\varepsilon$ be defined at \eqref{Eqconepsilon} and assume that it satisfies  \eqref{A2}---\eqref{A'4} and \eqref{A2primo}---\eqref{A5primo}  with exponents $p_i$ verifying \eqref{crescitapi}. 
    Let  $u_\varepsilon \in W^{1,p_n}_{\mathrm{loc}}(\Omega) $ be a local minimizer of the functional \eqref{functional2}.
    Then 
    we have  $$V_{p_i}((u_\varepsilon)_{x_i}) \in W^{1,2}_{\mathrm{loc}}(\Omega), \quad \forall i \in \{1,\dots, n\} .$$
  
\end{lem}
\begin{proof}
Fix a ball $B_R \Subset \Omega$ and consider radii $s<\rho <t<  r  <R$, a cut-off function $\eta \in C_0^{\infty}(B_t)$, with $\eta=1$ on $B_{\rho}$, $0 \leq \eta \leq 1$, $|D \eta | \leq \frac{c}{t-\rho}$, $|D^2 \eta | \leq \frac{c}{(t-\rho)^2}$  and $|h|\leq \frac{r - t}{2}$. We test the   Euler-Lagrange equation of the functional $\mathcal{F_\varepsilon}$  with the function
$$\varphi = \tau_{j, -h} \left(  \eta^2 \tau_{j,h} u_\varepsilon \right) $$
thus obtaining
\begin{equation*}
    \int_{\Omega} \Big\langle D_{\xi}F (x, D u_\varepsilon(x))+ \varepsilon(1+ |D u_\varepsilon|^2)^{\frac{p_n-2}{2}}D u_\varepsilon, \tau_{j,-h} D  \left(  \eta^2 \tau_{j,h} u_\varepsilon \right) \Big\rangle \, \dx = 0 \, , 
\end{equation*}
and hence
\begin{eqnarray*}\label{IntEeqlemm}
    0 &=&\int_{\Omega} \Big\langle  D_{\xi}F (x, D u_\varepsilon(x))+ \varepsilon(1+ |D u_\varepsilon|^2)^{\frac{p_n-2}{2}}D u_\varepsilon, \tau_{j,-h} \left( \eta^2 \tau_{j,h} D u_\varepsilon \right) \Big\rangle \, \dx  \cr\cr 
   &&+\int_{\Omega} \Big\langle D_{\xi}F (x, D u_\varepsilon(x))+ \varepsilon(1+ |D u_\varepsilon|^2)^{\frac{p_n-2}{2}}D u_\varepsilon , \tau_{j,-h} \left( 2 \eta D \eta \tau_{j,h} (u_\varepsilon)  \right) \Big\rangle \, \dx.
\end{eqnarray*}
By \textit{(ii)} in Proposition \ref{rapportoincrementale}, we have
\begin{eqnarray*}\label{IntEqlemm}
    0 &=&\int_{\Omega} \Big\langle \tau_{j,h} \left(D_{\xi}F (x, D u_\varepsilon(x))+ \varepsilon(1+ |D u_\varepsilon|^2)^{\frac{p_n-2}{2}}D u_\varepsilon\right),  \eta^2 \tau_{j,h} D u_\varepsilon \Big\rangle \, \dx  \cr\cr 
   &&+\int_{\Omega} \Big\langle D_{\xi}F (x, D u_\varepsilon(x))+ \varepsilon(1+ |D u_\varepsilon|^2)^{\frac{p_n-2}{2}}D u_\varepsilon, \tau_{j,-h} \left(2 \eta D \eta \tau_{j,h} (u_\varepsilon) \right)  \Big\rangle \, \dx,
\end{eqnarray*}
that can be written as follows
\begin{align*}\label{s}
 0=\int_{\Omega} &\langle D_{\xi}F(x+e_jh, D u_\varepsilon(x+e_jh))- D_{\xi}F (x+e_jh, D u_\varepsilon(x)) , \eta^2  \tau_{j,h} D u_\varepsilon \rangle \, \dx \nonumber \\
    &+  \int_{\Omega} \langle D_{\xi}F (x+e_jh, D u_\varepsilon(x))- D_{\xi}F (x, D u_\varepsilon(x)) , \eta^2  \tau_{j,h} D u_\varepsilon \rangle \, \dx \nonumber \\
      &+ \varepsilon  \int_{\Omega} \langle {\tau_{j, h} \left((1+ |D u_\varepsilon(x)|^2)^\frac{p_n-2}{2}D u_\varepsilon(x) \right)},  \eta^2  \tau_{j,h} D u_\varepsilon \rangle \, \dx \nonumber \\
       &+ 2\int_{\Omega} \Big\langle D_{\xi}F (x, D u_\varepsilon(x)) , \tau_{j,-h} \left(  \eta D \eta \tau_{j,h} (u_\varepsilon)  \right) \Big\rangle \, \dx \nonumber \\
    &+ 2 \varepsilon  \int_{\Omega} \langle  (1+ |D u_\varepsilon(x)|^2)^\frac{p_n-2}{2}D u_\varepsilon(x) , \tau_{j,-h} \left(  \eta D \eta \tau_{j,h} (u_\varepsilon)  \right) \Big\rangle \, \dx \\ \nonumber 
    & =: J_1 + J_2 + J_3 + J_4 + J_5
\end{align*}
and so
\begin{equation}\label{SommaJ}
   J_1 +J_3 \leq |J_2|  + |J_4| +|J_5| .
\end{equation}
 Since $F(x,\xi)$ satisfies  \eqref{A2}, by Lemma \ref{VpAndrea} we infer
\begin{eqnarray}\label{J_1}
    J_1  &\geq& \ell \int_{\Omega} \eta^2 \sum_{i=1}^{n}\, \left( \,\mu^2 + |(u_\varepsilon)_{x_i}(x+e_jh)|^2+|(u_\varepsilon)_{x_i}(x) |^2 \, \right)^\frac{p_i -2}{2} \left| \tau_{j,h} \Big( (u_\varepsilon)_{x_i}(x) \Big) \right|^2 \dx\cr\cr
    &\ge& c(p_i)\ell  \int_{\Omega} \eta^2 \sum_{i=1}^{n}\, \left|\tau_{j,h}\Big( V_{p_i}((u_\varepsilon)_{x_i})\Big) \right|^2\dx=: c(p_i)\ell\, \textbf{\textbf{RHS}}.
\end{eqnarray}
The integral $J_3$ can be estimated 
using the second inequality in Lemma \ref{VpAndrea} 
as follows
\begin{eqnarray}\label{J5}
    J_3 &\geq& \frac{\varepsilon}{c(p_n)} \int_{\Omega} \eta^2 (1+ |D u_\varepsilon(x+ he_j)|^2 + |D u_\varepsilon(x)|^2)^\frac{p_n-2}{2} \Big\lvert \tau_{j, h}( D u_\varepsilon(x)) \Big\rvert^2\dx \cr\cr
    &\geq& \frac{\varepsilon}{c(p_n)} \int_{\Omega} \eta^2 \left|\tau_{j, h}\Big(V_{p_n}( D u_\varepsilon(x))\Big) \right|^2\dx=: \frac{\varepsilon}{c(p_n)} \overline{\textbf{\textbf{RHS}}}.
\end{eqnarray}
By virtue of \eqref{A'4}, Young's inequality and Lemma \ref{ldiff}, we obtain
\begin{align}\label{J_3}
     |J_2| 
    & \leq  K |h|\, \, \sum_{i=1}^{n}  \int_{\Omega} \eta^2  \left| \tau_{j,h} \Big( (u_\varepsilon)_{x_i} \Big) \right| \, \left( \mu^2+|(u_\varepsilon)_{x_i}(x+e_jh)|^2+ |(u_\varepsilon)_{x_i}(x)|^2 \right)^\frac{p_i -1}{2}\dx \notag \\
    & \leq \beta \int_{\Omega} \eta^2 \sum_{i=1}^{n}\, \left( \,\mu^2+ |(u_\varepsilon)_{x_i}(x+e_j h)|^2+|(u_\varepsilon)_{x_i}(x) |^2 \, \right)^\frac{p_i -2}{2} \left| \tau_{j,h} \Big( (u_\varepsilon)_{x_i} \Big) \right|^2 \dx  \notag \\
    & \qquad + c_{\beta} |h|^2 \, \sum_{i=1}^{n}  \int_{B_t} \eta^2 \left( \mu^2\, +|(u_\varepsilon)_{x_i}(x+e_j h)|^2+|(u_\varepsilon)_{x_i}(x) |^2 \, \right)^\frac{p_i }{2} \,  \dx \notag \\
     & \leq  \beta \textbf{\textbf{RHS}} + c_{\beta} |h|^{2} (1+ \Vert Du_\varepsilon \Vert_{L^{\infty}(B_R)})^{p_n},
\end{align}
where, in the last estimate, we used Theorem \ref{thm2Mascolo}.
 By \eqref{A3'} , we infer 
 \begin{align}\label{J2lemma}
    | J_4| &\leq  2L\int_{\Omega}  
    \sum_{i=1}^{n} 
        \left(\mu^2 + |(u_\varepsilon)_{x_i}|^2\right)^{\frac{p_i-1}{2}}
        \left| \tau_{j,-h} \!\left( \eta \, \eta_{x_i} \, \tau_{j,h} u_\varepsilon \right) \right| 
    \dx .
\end{align}
Note that 
by \textit{(iii)} of Proposition \ref{rapportoincrementale}, we have
\begin{align}\label{tauIntI2}
\bigl| \tau_{j,-h}\bigl[\eta(x)\eta_{x_i} (x) \cdot \tau_{j,h}u_ \varepsilon(x)\bigr] \bigr| 
&\leq \bigl| \tau_{j,-h}\eta(x) \cdot \eta_{x_i}(x - h e_j) \cdot \tau_{j,h}u_\varepsilon(x - h e_j) \bigr| \notag \\
&\quad + \bigl| \eta(x)\,\tau_{j,-h} \eta_{x_i}(x) \cdot \tau_{j,h}u_\varepsilon(x - h e_j) \bigr| \notag \\
&\quad + \bigl| \eta(x) \eta_{x_i}(x) \cdot\tau_{j,-h}\tau_{j,h}u_\varepsilon(x) \bigr| \notag \\
&\leq \frac{c |h|} {(t - \rho)^2} \bigl|\tau_{j,h}u_\varepsilon(x - h e_j)\bigr| 
+ \frac{c } {(t - \rho)}  \bigl|\tau_{j,-h}\tau_{j,h}u_\varepsilon(x)\bigr| \notag \\
&= \frac{c |h|} {(t - \rho)^2} \bigl|\tau_{j,-h}u_\varepsilon(x )\bigr|
+ \frac{c } {(t - \rho)}  \bigl|\tau_{j,-h}\tau_{j,h}u_\varepsilon(x)\bigr|.
\end{align}
Inserting \eqref{tauIntI2} in \eqref{J2lemma} and recalling the properties of \( \eta \), we get
\begin{align}
   | J_4|&\leq   \frac{c |h|} {(t - \rho)^2} \int_{ B_r }  
    \sum_{i=1}^{n} 
        \left(\mu^2 + |(u_\varepsilon)_{x_i}|^2\right)^{\frac{p_i-1}{2}}
       \bigl|\tau_{j,-h}u_\varepsilon(x )\bigr|  \dx \notag \\
        &\qquad+  \frac{c } {(t - \rho)}\sum_{i=1}^{n} \int_{ B_t }
        \left(\mu^2 + |(u_\varepsilon)_{x_i}|^2\right)^{\frac{p_i-1}{2}}
        \bigl|\tau_{j,-h}\tau_{j,h}u_\varepsilon(x)\bigr|  \dx \notag \\
        & \leq  \frac{c|h|^2}{(t-\rho)^2}\,  (1+ \lvert \lvert   Du_\varepsilon \rvert \rvert_{L^\infty(B_R)})^{p_n-1} \left(\int_{ B_R }   \, |(u_\varepsilon)_{x_j}|^{p_j}\dx \right)^{\frac{1}{p_j}} \notag\\
         &\qquad+  \frac{c |h|} {(t - \rho)}(1+ \lvert \lvert   Du_\varepsilon \rvert \rvert_{L^\infty(B_R)})^{p_n-1}
       \left(\int_{ B_r } 
        \bigl|\tau_{j,h}(u_\varepsilon)_{x_j}\bigr|^{p_j}  \dx\right)^{\frac{1}{p_j}}  \notag \\
        &\leq \frac{c|h|^2}{(t-\rho)^2}\,  (1+ \lvert \lvert   Du_\varepsilon \rvert \rvert_{L^\infty(B_R)})^{p_n}  \notag \\
       &\qquad+  \frac{c |h|} {(t - \rho)} \,  (1+ \lvert \lvert   Du_\varepsilon \rvert \rvert_{L^\infty(B_R)})^{p_n-1} 
       \left(\int_{ B_r } 
        \bigl|\tau_{j,h}(u_\varepsilon)_{x_j}\bigr|^{p_j}  \dx\right)^{\frac{1}{p_j}}, 
\end{align}
where we used Theorem \ref{thm2Mascolo}, H\"older's inequality and  Lemma \ref{ldiff}.\\
From Lemma \ref{VpAndrea}, H\"older's and  Young's inequalities and using   Theorem \ref{thm2Mascolo} again, we infer
\begin{align}\label{J2primofinale}
      |  J_4|& \leq \frac{c|h|^2}{(t-\rho)^2}\,  (1+ \lvert \lvert   Du_\varepsilon \rvert \rvert_{L^\infty(B_R)})^{p_n}+  \frac{c |h|} {(t - \rho)} (1+ \lvert \lvert   Du_\varepsilon \rvert \rvert_{L^\infty(B_R)})^{p_n-1}  \notag \\
       &\cdot \left(\int_{ B_r } |\tau_{j,h} \left( V_{p_j}(u_{x_j}) \right)|^{p_j}
    \left( \mu^2 + |u_{x_j}(x+e_jh)|^2 + |u_{x_j}(x)|^2 \right)^{p_j\frac{2-p_j}{4}}   \dx\right)^{\frac{1}{p_j}} \notag \\
    & \leq \frac{c|h|^2}{(t-\rho)^2}\,  (1+ \lvert \lvert   Du_\varepsilon \rvert \rvert_{L^\infty(B_R)})^{p_n}+  \frac{c |h|} {(t - \rho)} (1+ \lvert \lvert   Du_\varepsilon \rvert \rvert_{L^\infty(B_R)})^{p_n-1}  \notag \\
       &\cdot  \left(\int_{ B_r } |\tau_{j,h} \left( V_{p_j}(u_{x_j}) \right)|^{2}\dx\right)^{\frac{1}{2}}
    \left(\int_{ B_r } \left( \mu^2 + |u_{x_j}(x+e_jh)|^2 + |u_{x_j}(x)|^2 \right)^{\frac{p_j}{2}}   \dx\right)^{\frac{2-p_j}{2p_j}} \notag \\
    & \leq \frac{c|h|^2}{(t-\rho)^2}\,  (1+ \lvert \lvert   Du_\varepsilon \rvert \rvert_{L^\infty(B_R)})^{p_n}   + \sigma \int_{ B_r } |\tau_{j,h} \left( V_{p_j}(u_{x_j}) \right)|^{2} \, \dx \notag \\
& + \frac{c |h|^2} {(t - \rho)^2}(1+ \lvert \lvert   Du_\varepsilon \rvert \rvert_{L^\infty(B_R)})^{2p_n-2}  \left(\int_{ B_r }  \left(\mu^2 + |u_{x_j}(x+e_jh)|^2 + |u_{x_j}(x)|^2\right)^{\frac{p_j}{2}} \dx\right)^{\frac{2-p_j}{p_j}} \notag \\
 & \leq \frac{c |h|^2} {(t - \rho)^2}(1+ \lvert \lvert   Du_\varepsilon \rvert \rvert_{L^\infty(B_R)})^{2p_n} + \sigma \sum_{i=1}^n\int_{B_r } |\tau_{i,h} \left( V_{p_i}(u_{x_i}) \right)|^{2} \, \dx.
\end{align}
Thanks to Theorem \ref{thm2Mascolo}, estimate \eqref{tauIntI2}, Lemmas \ref{ldiff} and  \ref{VpAndrea}, we have
\begin{align}\label{J6seconda}
|J_5| 
& \leq 2 \varepsilon  \int_{ B_t } \left|(1+ |D u_\varepsilon(x)|^2)^\frac{p_n-1}{2} \right| \, \Big| \tau_{j,- h} \left(\eta \, D \eta \, \tau_{j,h} u_\varepsilon \right) \Big| \, \dx  \notag \\
 &\leq c\varepsilon   (1+ \lvert \lvert   Du_\varepsilon \rvert \rvert_{L^\infty(B_R)})^{p_n-1}\int_{ B_t }  |\tau_{j, -h} \left(\eta \, D \eta \, \tau_{j,h} u_\varepsilon \right)|  \, \dx  \notag \\
 & \leq \frac{c\varepsilon  |h| }{(t-\rho)^2}(1+ \lvert \lvert   Du_\varepsilon \rvert \rvert_{L^\infty(B_R)})^{p_n-1}  \int_{ B_t }  \bigl|\tau_{j,-h}(u_\varepsilon)(x )\bigr| \dx \notag \\
& \qquad + \frac{c\varepsilon   }{(t-\rho)}(1+ \lvert \lvert   Du_\varepsilon \rvert \rvert_{L^\infty(B_R)})^{p_n-1}  \int_{ B_t }   \bigl|\tau_{j,-h}\tau_{j,h}(u_\varepsilon)(x)\bigr|\dx \notag \\
&  \leq \frac{c\varepsilon  |h|^2 }{(t-\rho)^2}(1+ \lvert \lvert   Du_\varepsilon \rvert \rvert_{L^\infty(B_R)})^{p_n-1}  \left(\int_{B_R} \bigl| (u_\varepsilon)_{x_j}\bigr|^{p_j} \dx\right)^{\frac{1}{p_j}} \notag \\
& \qquad + \frac{c\varepsilon  |h|}{(t-\rho)} (1+ \lvert \lvert   Du_\varepsilon \rvert \rvert_{L^\infty(B_R)})^{p_n-1} \left(\int_{ B_r }  \bigl|\tau_{j,h}\left( (u_\varepsilon)_{x_j}(x) \right) \bigr|^{p_j}\dx\right)^{\frac{1}{p_j}} \notag \\
&  \leq \frac{c\varepsilon  |h|^2}{(t-\rho)^2}(1+ \lvert \lvert   Du_\varepsilon \rvert \rvert_{L^\infty(B_R)})^{p_n} +  \frac{c\varepsilon  |h| }{(t-\rho)}(1+ \lvert \lvert   Du_\varepsilon \rvert \rvert_{L^\infty(B_R)})^{p_n-1}\cdot\notag \\
& \qquad\cdot  \left(\int_{ B_r }|\tau_{j,h}V_{p_j}((u_\varepsilon)_{x_j})|^{p_j}
    \left( \mu^2 + |(u_\varepsilon)_{x_j}(x+e_jh)|^2 + |(u_\varepsilon)_{x_j}(x)|^2 \right)^{p_j\frac{2-p_j}{4}}   \dx\right)^{\frac{1}{p_j}} \notag \\
    & \leq \frac{c|h|^2}{(t-\rho)^2}\,  (1+ \lvert \lvert   Du_\varepsilon \rvert \rvert_{L^\infty(B_R)})^{p_n}+  \frac{c |h|} {(t - \rho)} (1+ \lvert \lvert   Du_\varepsilon \rvert \rvert_{L^\infty(B_R)})^{p_n-1}  \notag \\
       &\cdot  \left(\int_{{ B_r }} |\tau_{j,h} \left( V_{p_j}(u_{x_j}) \right)|^{2}\dx\right)^{\frac{1}{2}}
    \left(\int_{{ B_r }} \left( \mu^2 + |u_{x_j}(x+e_jh)|^2 + |u_{x_j}(x)|^2 \right)^{\frac{p_j}{2}}   \dx\right)^{\frac{2-p_j}{2p_j}} \notag \\
&  \leq \frac{c\varepsilon  |h|^2 }{(t-\rho)^2}  (1+ \lvert \lvert   Du_\varepsilon \rvert \rvert_{L^\infty(B_R)})^{p_n}  + \sigma \int_{{ B_r }} |\tau_{j,h}V_{p_j}((u_\varepsilon)_{x_j})|^{2} \, \dx \notag \\
& \qquad +  \frac{c\varepsilon  |h|^2 }{(t-\rho)^2}(1+ \lvert \lvert   Du_\varepsilon \rvert \rvert_{L^\infty(B_R)})^{2p_n-2}  \left(\int_{ B_r }  \left(\mu^2 + |u_{x_j}(x+e_jh)|^2 + |u_{x_j}(x)|^2\right)^{\frac{p_j}{2}} \dx\right)^{\frac{2-p_j}{p_j}}\notag \\
&  \leq  \sigma \int_{ B_r } |\tau_{j,h}V_{p_j}((u_\varepsilon)_{x_j})|^{2} \, \dx   +  \frac{c\varepsilon  |h|^2 }{(t-\rho)^2}(1+ \lvert \lvert   Du_\varepsilon \rvert \rvert_{L^\infty(B_R)})^{2p_n},
\end{align}
where we also used Young's inequality.\\
Inserting \eqref{J_1}, \eqref{J5}, \eqref{J_3}, \eqref{J2primofinale}, \eqref{J6seconda}  in \eqref{SommaJ}, we obtain
\begin{align}\label{sum}
 c(p_i)\ell  \,{\textbf{\textbf{RHS}}}+ \frac{\varepsilon}{c} \overline{\textbf{\textbf{RHS}}}  & \leq  \beta \, {\textbf{\textbf{RHS}}} + 2\sigma \sum_{i=1}^n\int_{ B_r } |\tau_{i,h}V_{p_i}((u_\varepsilon)_{x_i})|^{2} \, \dx \notag\\
 &\qquad  + \frac{c  |h|^2 }{(t-\rho)^2}(1+ \lvert \lvert   Du_\varepsilon \rvert \rvert_{L^\infty(B_R)})^{2p_n}. 
\end{align}
 Choosing $\beta = \frac{c(p_i)\ell}{2}$, reabsorbing the first  term in the right-hand side  by the left-hand side of previous inequality and neglecting the non negative term $\overline{\textbf{RHS}}$ , we get
\begin{align}\label{stimaUnioneJ}
   \frac{c(p_i)\ell}{2} \int_{B_\rho}  \sum_{i=1}^{n}\, \Big|\tau_{j,h}\Big( V_{p_i}((u_\varepsilon)_{x_i})\Big)\Big|^2\dx &\leq 2\sigma \sum_{i=1}^n\int_{ B_r } \Big|\tau_{i,h}\Big(V_{p_i}((u_\varepsilon)_{x_i})\Big)\Big|^{2} \, \dx\notag \\
&  +  \frac{c  |h|^2 }{(t-\rho)^2}(1+ \lvert \lvert   Du_\varepsilon \rvert \rvert_{L^\infty(B_R)})^{2p_n} ,
\end{align}
where we used that  $\eta=1$ on $B_\rho$.
This  in particular yields, for every ${j}=1, \dots, n$,
\begin{align}
  \int_{B_{\rho}}\Big|\tau_{j,h}\Big(V_{p_j}((u_\varepsilon)_{x_j})\Big)\Big|^{2} \, \dx & \leq \frac{4 \sigma}{c(p_i)\ell} \sum_{i=1}^n\int_{ B_r } \Big|\tau_{i,h}\Big(V_{p_i}((u_\varepsilon)_{x_i})\Big)\Big|^{2} \, \dx \notag \\
& \qquad  +   \frac{c |h|^2 }{(t-\rho)^2}(1+ \lvert \lvert   Du_\varepsilon \rvert \rvert_{L^\infty(B_R)})^{2p_n}. 
  \end{align}
  Summing over $j=1,\dots, n$,  we infer
  \begin{align}\label{lemma31tau}
 \sum_{j=1}^n \int_{B_{\rho}}|\tau_{j,h}\Big(V_{p_j}((u_\varepsilon)_{x_j})\Big)|^{2} \, \dx & \leq \frac{4 \sigma n}{c(p_i)\ell}\sum_{i=1}^n\int_{ B_r } |\tau_{i,h}\Big(V_{p_i}((u_\varepsilon)_{x_i})\Big)|^{2} \, \dx \notag \\
& \qquad  +  \frac{c  |h|^2 }{(t-\rho)^2}(1+ \lvert \lvert   Du_\varepsilon \rvert \rvert_{L^\infty(B_R)})^{2p_n}  
  \end{align}
   Since previous estimate holds for all radii $s< \rho <t<r<R$ with constants independent of the radii, in particular it applies choosing
  $$t= \rho + \frac{r-\rho}{2},$$
  so that  $t-\rho=\frac{r- \rho}{2}$.  Hence, 
choosing $\sigma >0$ such that $ \frac{4 \sigma n}{c(p_i)\ell}<1$ and setting $$\phi(\rho)=\sum_{i=1}^n \int_{B_{\rho}}\left|\tau_{i,h}\Big(V_{p_i}((u_\varepsilon)_{x_i})\Big)\right|^{2} \, \dx, $$
inequality \eqref{lemma31tau} can be written as follows
    \begin{equation}\label{phirholemma31}
\phi(\rho) \leq \frac{4 \sigma n}{c(p_i)\ell} \phi({r})+  \frac{c|h|^2}{{(r-\rho)^2}}
\end{equation} 
Since \eqref{phirholemma31} holds true for every radii $s<\rho <t<r< R$,  we are legitimate to apply Lemma \ref{lm3} to the function $\phi: [s,R] \to \mathbb{R}$, thus getting 
    \begin{align}\label{lemma31tauii}
      \int_{B_s}\left|\tau_{i,h}\Big(V_{p_i}((u_\varepsilon)_{x_i})\Big)\right|^{2} \, \dx & \leq c |h|^2 ,
\end{align}
for a positive constant $c = c(n,p_i,\rho, R, \lambda, K,\Lambda, \lvert \lvert Du_\varepsilon \rvert \rvert_{\infty})$. Hence by Lemma \ref{Lemmahzero} we have
 $$V_{p_i}((u_\varepsilon)_{x_i}) \in W^{1,2}_{\mathrm{loc}}(\Omega), \qquad \forall i=1, \dots, n,$$ 
 which implies in particular that $$|(u_\varepsilon)_{x_i x_j}|^2|(u_\varepsilon)_{x_i}|^{p_i-2} \in L^1_{\mathrm{loc}}(\Omega), \qquad \forall i=1, \dots, n \text{ and  }  \forall j=1, \dots, n,
$$i.e. the conclusion.  \qedhere
\end{proof}

\section{The a priori estimate}\label{Apriorii}
The aim of this section is to establish an a priori estimate for the minimizers of the functional \eqref{functional} and this constitutes the core  of the proof of the main theorem. We shall use the  difference quotients method  combined with the interpolation inequality of Proposition \ref{LemmaBrasco}, which will be essential in estimating the terms involving the coefficients.
\begin{thm}\label{AppThm}
 Let $u \in W^{1,\mathbf{p}}(\Omega)$ be a local minimizer of \eqref{functional} under assumptions \eqref{crescitapi},\eqref{pi}, \eqref{A0}--\eqref{A4}. 
 Moreover, assume that 
 \begin{equation}\label{assp_n}
    p_n<\bar{p}^* ,
 \end{equation}
 with $\bar{p}^*$ defined at \eqref{definizionePi} and that the function $g$ appearing in \eqref{A4} belongs to $L^r_{\mathrm{loc}}(\Omega)$, with an exponent $r$ such that 
\begin{equation}\label{Ipotesir}
r> p_n + 2.
\end{equation}
 If 
\begin{equation}\label{Apriori}
V_{p_i}(u_{x_i}) \in {W^{1,2}_{\mathrm{loc}}(\Omega)}, \qquad\forall i=1,\dots,n
\end{equation}
then the following estimates
\begin{equation}\label{StimaTeo1}
      \sum_{i=1}^{n}\! \int_{B_{\rho}}\!\!\!\! \left( \mu^2 + |u_{x_i}|^2 \right)^{\frac{p_i-2}{2}} |u_{x_i x_j}|^2 \dx   \leq c  \left( \sum_{i=1}^{n} \Vert u_{x_i} \Vert_{L^{p_i}(B_R)} + \Vert g \Vert_{L^r (B_{R})} \right)^\sigma\!\!  
\end{equation}
and
\begin{align}\label{uxistima}
    \sum_{i=1}^{n}  \int_{B_{\rho}} \left( \mu^2 +|u_{x_i}|^2 \right)^\frac{p_i +2 }{2} \dx  & \leq c  \left( \sum_{i=1}^{n} \Vert u_{x_i} \Vert_{L^{p_i}(B_R)} + \Vert g \Vert_{L^r (B_{R})} \right)^\sigma 
\end{align}
hold for every pair of concentric balls $B_{\rho} \subset B_{R} \Subset \Omega$, where $c = c(n,p_i,\ell, L,\rho,R)$ and $\sigma= \sigma (n,p_i)$ are positive constants.
\end{thm}
\begin{proof} Before entering in the core of the proof, it is worth noticing that by virtue of \eqref{assp_n}, Theorem \ref{ulimitatoCupini}, which applies to the minimizer of \eqref{functional} by virtue of our assumptions on $f(x,\xi)$,  implies that $u\in L^\infty_{\loc}(\Omega)$. Hence, assumption \eqref{Apriori} allows us to use Proposition \ref{InterpBrasco} to deduce that
\begin{equation}\label{highint}
u_{x_i}\in L^{p_i+2}_{\loc}(\Omega)\qquad \text{for every } i=1,\dots,n
\end{equation}
together with estimate \eqref{LemmaBr}.\\
Fix a ball $B_R \Subset \Omega$ and consider radii $\rho<\varrho< s<t<t'<\rho'<R'< R$, a cut-off function $\eta \in C_0^{\infty}(B_t)$, with $\eta=1$ on $B_{s}$, $0 \leq \eta \leq 1$, $|D \eta | \leq \frac{c}{t-s}$  and $|h|\leq \frac{t' - t}{2}$.
We test the Euler-Lagrange equation of \eqref{functional} with the function
$$\varphi = \tau_{j, -h} \left(  \eta^2 \tau_{j,h} u \right) $$
thus obtaining
\begin{equation*}
    \int_{\Omega} \Big\langle f_{\xi} (x, D u(x)), \tau_{j,-h} D  \left(  \eta^2 \tau_{j,h} u \right) \Big\rangle \, \dx = 0 \, , 
\end{equation*}
and hence
\begin{eqnarray*}
   0&=& \int_{\Omega} \Big\langle f_{\xi} (x, D u(x)), \tau_{j,-h}   \left(  \eta^2  \tau_{j,h} Du  \right) \Big\rangle \, \dx  \cr\cr
   && \qquad +2 \int_{\Omega} \Big\langle f_{\xi} (x, D u(x)), \tau_{j,-h}   \left( \eta  \eta_{x_i} \tau_{j,h} u \right) \Big\rangle \, \dx .
\end{eqnarray*}
By $(ii)$ of Proposition \ref{rapportoincrementale},
\begin{align*}
   0= &\int_{\Omega} \Big\langle \tau_{j,h} f_{\xi} (x, D u(x),\eta^2  \tau_{j,h} Du  \Big\rangle \, \dx  \notag \\
   & \qquad +2 \int_{\Omega} \Big\langle f_{\xi} (x, D u(x)), \tau_{j,-h}   \left( \eta  \eta_{x_i} \tau_{j,h} u \right) \Big\rangle \, \dx \notag
\end{align*}
We  may rewrite the previous equality  as follows
\begin{align}\label{SommaIntera}
 0= &\int_{\Omega} \langle f_{\xi} (x+e_jh, Du(x+e_jh))- f_{\xi} (x+e_jh, Du(x)) , \eta^2  \tau_{j,h} Du \rangle \, \dx \notag \\
  &+  \int_{\Omega} \langle f_{\xi} (x+e_jh, Du(x))- f_{\xi} (x, Du(x)) , \eta^2  \tau_{j,h} Du \rangle \, \dx \notag \\
    &+ 2\int_{\Omega} \Big\langle f_{\xi} (x, D u(x)), \tau_{j,-h}   \left( \eta  \eta_{x_i} \tau_{j,h} u \right) \Big\rangle dx  \notag \\
    & =: I_1 + I_2 + I_3 , \notag
\end{align}
 that yields
\begin{equation}\label{SommaInt}
   I_1 \leq |I_2|  + |I_3| .
\end{equation}
Assumption \eqref{A2} together with Lemma \ref{VpAndrea} implies
\begin{align}\label{I1}
    I_1 
    & \geq \ell \int_{\Omega} \eta^2 \sum_{i=1}^{n}\, \left( \mu^2+ \, |u_{x_i}(x+e_jh)|^2+|u_{x_i}(x) |^2 \, \right)^\frac{p_i -2}{2}| \tau_{j,h} u_{x_i} |^2 \dx \nonumber\\
    &\geq \ell \int_{\Omega} \eta^2 \sum_{i=1}^{n}\, \left| \tau_{j,h} \Big(V_{p_i}(u_{x_i})\Big) \right|^2 \dx =: \ell \, \textbf{\textbf{RHS}}.
\end{align}
From now on, to simplify the presentation, we will use the notation
$$\mathcal{U}_{i,j}(h)=:
\mu^2 + \, |u_{x_i}(x+e_jh)|^2+|u_{x_i}(x) |^2,  
$$
and, when needed, we will abbreviate writing $\mathcal{U}_{j}$ in place of $\mathcal{U}_{j,j}$.  \\
By assumption \eqref{A4}, Young's and  H\"older's    inequalities,  Lemma \ref{ldiff} and the properties of $\eta$,
 we derive
\begin{align}\label{IntI3prima}
     |I_2| 
    & \leq c |h| \sum_{i=1}^{n}  \int_{\Omega} \eta^2 \, (g(x+he_j)+ g(x))\, \mathcal{U}_{i,j}(h)^\frac{p_i -1}{2}|\tau_{j, h}u_{x_i}|  \dx \notag \\
    & \leq \varepsilon \ \ \ \ {\sum_{i=1}^{n}\int_{B_t} \eta^2 \, \mathcal{U}_{i,j}(h)^\frac{p_i -2}{2}| \tau_{j,h} u_{x_i} |^2 \dx}  \notag \\
    & \qquad + c_{\varepsilon} |h|^2 \sum_{i=1}^{n}  \int_{B_t} \eta^2\, \mathcal{U}_{i,j}(h)^{\frac{p_i}{2}} \, (g(x+he_j)+ g(x))^2\,  \dx \notag \\
     & \leq \varepsilon \textbf{\textbf{RHS}} + c_{\varepsilon} |h|^{2}  \left( \int_{B_{R}} g(x)^{r} \dx \right)^{\frac{2}{r}}  \sum_{i=1}^{n}\left( \int_{B_{t}}  \mathcal{U}_{i,j}(h)^{\frac{p_i r}{2(r-2)}} \dx \right)^{\frac{r-2}{r}}\notag \\
     & \leq \varepsilon \textbf{\textbf{RHS}} + c_{\varepsilon} |h|^{2}  \left( \int_{B_{R}} g(x)^{r} \dx \right)^{\frac{2}{r}}  \sum_{i=1}^{n}\left( \int_{B_{t'}}  \left(\mu^2 +|u_{x_i} |^2 \, \right)^{\frac{p_i r}{2(r-2)}} \dx \right)^{\frac{r-2}{r}},
\end{align}
for a constant $\varepsilon>0$ that will be chosen later.\\
Note now that
$$\frac{(p_i+2)(r-2)}{p_ir}>1 \, \text{ iff } \, r>p_i+2,$$
that holds true by assumption \eqref{Ipotesir} for every $i=1, \dots,n$.
Hence it is legitimate to apply H\"older's inequality in the last term of \eqref{IntI3prima}, to infer 
\begin{align}\label{I3finaleapriori}
 |I_2| &\leq  \varepsilon \textbf{\textbf{RHS}} +    c_{\varepsilon} |h|^{2}   \left( \int_{B_{R}} g(x)^{r} \dx \right)^{\frac{2}{r}} \, \sum_{i=1}^{n}  \left( \int_{B_{t'}} \left(\mu^2 +|u_{x_i} |^2 \, \right)^{\frac{p_i+2}{2}} \dx \right)^{\frac{p_i}{p_i +2}} \notag \\
 & \leq \varepsilon \textbf{\textbf{RHS}}+ \beta |h|^2  \sum_{i=1}^{n}   \int_{B_{{t'}}} \left(\mu^2 +|u_{x_i} |^2 \, \right)^{\frac{p_i+2}{2}} \dx \notag\\
 &\qquad+  c_{\varepsilon, \beta} |h|^2 \sum_{i=1}^{n} \left( \int_{B_{R}} g(x)^{r} \, \dx \right)^{\frac{p_i+2}{r}} ,
\end{align}
where we used Young's inequality with exponents $ \frac{p_i +2}{p_i}$ and $\frac{p_i+2}{2}$ and $\beta >0$ will be chosen later.\\
 To handle the term $I_3$ we use  assumption \eqref{A3'} and \eqref{tauIntI2} as follows
\begin{align}\label{I2secondo}
            |I_3|
& \leq c \int_{ \Omega }  
    \sum_{i=1}^{n} 
        \left(\mu^2 + |u_{x_i}|^2\right)^{\frac{p_i-1}{2}}
        \left| \tau_{j,-h} \!\left( \eta \, \eta_{x_i} \, \tau_{j,h} u \right) \right|
   \dx\\\nonumber
  & \leq  \frac{c |h|} {(t - \rho)^2} \int_{ B_{t+\frac{t'-t}{2}} } \,\,  
    \sum_{i=1}^{n} 
        \left(\mu^2 + |u_{x_i}|^2\right)^{\frac{p_i-1}{2}}\bigl|\tau_{j,-h}u\bigr|\dx\\\nonumber
        &+\frac{c } {(t - \rho)} \int_{ B_t }   
    \sum_{i=1}^{n} 
        \left(\mu^2 + |u_{x_i}|^2\right)^{\frac{p_i-1}{2}} \bigl|\tau_{j,-h}\tau_{j,h}u\bigr|\dx\\\nonumber
        &=I_3'+I_3'',
  \end{align}
  where we used that $|h|<\frac{t'-t}{2}$
By H\"older's inequality, we get
\begin{equation}
|I_3'| \le \frac{c |h|}{(t-s)^2}  \sum_{i=1}^{n}\left(\int_{  B_{t+\frac{t'-t}{2}} }  
        \left(\mu^2 + |u_{x_i}|^2\right)^{\frac{p_i}{2}}\,\dx\right)^{\frac{p_i-1}{p_i}}
        \left(\int_{  B_{t+\frac{t'-t}{2}}  }  \left| \tau_{j,-h} u \right|^{p_i}\,\dx\right)^{\frac{1}{p_i}}
\end{equation}
At this point, if $i\le j$ one immediately gets
\begin{eqnarray}\label{iminoredij}
|I_3'| &\le& \frac{c |h|}{(t-s)^2}  \sum_{i=1}^{n}\left(\int_{ B_{t'}}  
        \left(\mu^2 + |u_{x_i}|^2\right)^{\frac{p_i}{2}}\,\dx\right)^{\frac{p_i-1}{p_i}}
        \left(\int_{  B_{t+\frac{t'-t}{2}} } \left| \tau_{j,-h} u \right|^{p_j}\,\dx\right)^{\frac{1}{p_j}}\cr\cr
        &\le& \frac{c |h|^2}{(t-s)^2}  \sum_{i=1}^{n}\left(\int_{B_{t'}}  
        \left(\mu^2 + |u_{x_i}|^2\right)^{\frac{p_i}{2}}\,\dx\right)^{\frac{p_i-1}{p_i}}
        \left(\int_{B_{t'}} \left|  u_{x_j} \right|^{p_j}\,\dx\right)^{\frac{1}{p_j}}\cr\cr
        &\le& \frac{c |h|^2}{(t-s)^2}  \sum_{i=1}^{n}\left(\int_{B_{t'}}  
        \left(\mu^2 + |u_{x_i}|^2\right)^{\frac{p_i}{2}}\,\dx\right)^{\frac{p_i-1}{p_i}}
        \sum_{j=1}^{n}\left(\int_{B_{t'}} \left|  u_{x_j} \right|^{p_j}\,\dx\right)^{\frac{1}{p_j}}\cr\cr
        &\le&\frac{c |h|^2}{(t-s)^2}  \sum_{i=1}^{n}\int_{B_{t'}}  
        \left(\mu^2 + |u_{x_i}|^2\right)^{\frac{p_i}{2}}\,\dx,
\end{eqnarray}
where, in the second line of estimate, we used Lemma \ref{ldiff}.
On the other hand, if $i> j$, \eqref{pi} implies in particular that
$$p_j<p_i<p_j+2$$
and so there exists $\vartheta\in (0,1)$ such that
$$\frac{1}{p_i}=\frac{\vartheta}{p_j}+\frac{1-\vartheta}{p_j+2}.$$
By the interpolation inequality we then have
\begin{eqnarray}
    \label{imaggioredij}
|I_3'| &\le& \frac{c |h|}{(t-s)^2}  \sum_{i=1}^{n}\left(\int_{ B_{t+\frac{t'-t}{2}}}  
        \left(\mu^2 + |u_{x_i}|^2\right)^{\frac{p_i}{2}}\,\dx\right)^{\frac{p_i-1}{p_i}}\cr\cr
        &&\cdot\left(\int_{ B_{t+\frac{t'-t}{2}}} \left| \tau_{j,-h} u \right|^{p_j}\,\dx\right)^{\frac{\vartheta}{p_j}}\left(\int_{ B_{t+\frac{t'-t}{2}}} \left| \tau_{j,-h} u \right|^{p_j+2}\,\dx\right)^{\frac{1-\vartheta}{p_j+2}}\cr\cr 
        &\le &\frac{c |h|^2}{(t-s)^2}  \sum_{i=1}^{n}\left(\int_{B_{t'}}  
        \left(\mu^2 + |u_{x_i}|^2\right)^{\frac{p_i}{2}}\,\dx\right)^{\frac{p_i-1}{p_i}}\cr\cr
        &&\cdot\left(\int_{B_{t'}} \left|  u_{x_j} \right|^{p_j}\,\dx\right)^{\frac{\vartheta}{p_j}}\left(\int_{B_{t'}} \left| u_{x_j} \right|^{p_j+2}\,\dx\right)^{\frac{1-\vartheta}{p_j+2}}\cr\cr 
        &\le &\beta |h|^2\int_{B_{t'}} \left| u_{x_j} \right|^{p_j+2}\,\dx\cr\cr
        &+&\sum_{i=1}^{n}\frac{c |h|^2}{(t-s)^{\frac{4p_i}{p_i+p_j}}}  \left(\int_{B_{t'}}  
        \left(\mu^2 + |u_{x_i}|^2\right)^{\frac{p_i}{2}}\,\dx\right)^{\frac{2(p_i-1)}{p_i+p_j}}\cr\cr
        &&\cdot\left(\int_{B_{t'}} \left|  u_{x_j} \right|^{p_j}\,\dx\right)^{\frac{p_j+2-p_i}{p_i+p_j}}\cr\cr 
        &\le &\beta |h|^2\sum_{i=1}^n\int_{B_{t'}} (\mu^2+\left| u_{x_i} \right|^2)^{\frac{p_i+2}{2}}\,\dx\cr\cr
        &+&\sum_{i=1}^{n}\frac{c |h|^2}{(t-s)^{\frac{4p_i}{p_i+p_j}}}  \left(\int_{B_{t'}}  
        \left(\mu^2 + |u_{x_i}|^2\right)^{\frac{p_i}{2}}\,\dx\right)^{\frac{2(p_i-1)}{p_i+p_j}}\cr\cr
        &&\cdot\left(\int_{B_{t'}} \left|  u_{x_j} \right|^{p_j}\,\dx\right)^{\frac{p_j+2-p_i}{p_i+p_j}} \notag \\
 & \leq& \beta |h|^2\sum_{i=1}^n\int_{B_{t'}} (\mu^2+\left| u_{x_i} \right|^2)^{\frac{p_i+2}{2}}\,\dx\cr\cr
        &+& \sum_{i=1}^{n}\frac{c |h|^2}{(t-s)^{\frac{4p_i}{p_i+p_j}}}  \int_{ B_{t'}}  
        \left(\mu^2 + |u_{x_i}|^2\right)^{\frac{p_i}{2}}\,\dx , 
\end{eqnarray}
where we used in turn Lemma \ref{ldiff}, Young's inequality with exponents $\frac{p_j+2}{1-\vartheta}, \frac{p_j+2}{p_j+1+\vartheta}$ and the explicit expression of $\vartheta=\frac{p_j(p_j+2-p_i)}{2p_i}$. Note that the right hand side of \eqref{imaggioredij} is finite by virtue of \eqref{highint}.
Joining \eqref{iminoredij} and \eqref{imaggioredij}, we get
\begin{align}\label{I3'completoo}
    I_3' &\leq \beta |h|^2\sum_{i=1}^n\int_{B_{t'}} (\mu^2+\left| u_{x_i} \right|^2)^{\frac{p_i+2}{2}}\,\dx \notag \\
        & \qquad + \frac{c |h|^2}{(t-s)^{\alpha'}}  \sum_{i=1}^{n}\int_{ B_{t'}}   
        \left(\mu^2 + |u_{x_i}|^2\right)^{\frac{p_i}{2}}\,\dx
\end{align}
with $\alpha'=\alpha'(p_i).$
Now,  we take care of $I_3''$. Thanks to    H\"older's inequality with  exponents $ \left(\frac{p_i+2}{p_i-1}, \frac{ p_i+2}{3}\right)$ and Lemma \ref{ldiff}, we have
\begin{eqnarray}\label{I3''split}
     I_{3}''&\leq &\frac{c}{t - s} 
    \sum_{i=1}^{n}\left( \int_{B_t} (\mu^2 + |u_{x_i}|^2)^\frac{p_i+2}{2} \, \dx \right)^{\!\frac{p_i-1}{p_i+2}}
   \left( \int_{B_{t}} |\tau_{j,-h}\tau_{j,h}u|^{\frac{p_i+2}{3}} \, \dx \right)^{\frac{3}{p_i+2}} \cr\cr
    &\leq & \frac{c |h|}{t - s} 
    \sum_{i=1}^{n} 
    \left( \int_{B_t} (\mu^2 + |u_{x_i}|^2)^\frac{p_i+2}{2} \, \dx \right)^{\!\frac{p_i-1}{p_i+2}}
    \left( \int_{B_{t'}} |\tau_{j,h}u_{x_j}(x)|^{\frac{p_i+2}{3}} \, \dx \right)^{\frac{3}{p_i+2}}\cr\cr 
    &=:&\frac{c |h|}{t - s} 
    \sum_{i=1}^{n} 
    \left( \int_{B_t} (\mu^2 + |u_{x_i}|^2)^\frac{p_i+2}{2} \, \dx \right)^{\!\frac{p_i-1}{p_i+2}}
\cdot \mathbf{J}.
\end{eqnarray}
The estimate of $\mathbf{J}$ depends on the following cases 
$$ \frac{p_i+2}{3}\le p_j \qquad\,\,  \text{and}\,\,\qquad p_j< \frac{p_i+2}{3}.$$
In the first case, 
since $u_{x_j} \in L^{p_j}_{\mathrm{loc}}(\Omega)$, using H\"older's inequality, we get  
\begin{align}
    \mathbf{J}
    &\leq c
    \left(\int_{B_{t'}} |\tau_{j,h}u_{x_j}(x)|^{p_j} \, \dx \right)^{\frac{1}{p_j}}. \notag
    \end{align}
    and so, by Lemma \ref{VpAndrea}, we  obtain
 \begin{eqnarray}\label{Jpjmin}
   \mathbf{J} &\leq& c
    \left(\int_{B_{t'}}|\tau_{j,h}V_{p_j}(u_{x_j})|^{p_j}
    \mathcal{U}_{j}(h)^{\frac{(2-p_j)p_j}{4}} 
    \dx \right)^{\frac{1}{p_j}} \notag \\
    &\leq&  c \left( \int_{B_{t'}} |\tau_{j,h}V_{p_j}(u_{x_j})|^{2} \, \dx \right)^{\!\frac{1}{2}}\left( \int_{B_{t'}} \mathcal{U}_{j}(h)^{\frac{p_j}{2}} 
    \dx \right)^{\!\frac{2-p_j}{2p_j}},
\end{eqnarray}
where we used  H\"older’s inequality with exponents $(\frac{2}{p_j}, \frac{2}{2-p_j})$, which is legitimate if $p_j <2$. Of course if $p_j=2$ we have $|\tau_{j,h} u_{x_j}|^2= |\tau_{j,h}\left(V_2(u_{x_j}) \right)|^2$.
\\
Suppose now that ${p_j< \frac{p_i+2}{3}}$. 
From the definition of $\mathbf{J}$ and since by \eqref{pi} it holds $\frac{p_i+2}{3}\le \frac43$, applying  Lemma \ref{VpAndrea}, we obtain
    \begin{eqnarray}\label{j2apri}
   |\mathbf{J}| &\leq& 
    \left( \int_{B_{t'}} |\tau_{j,h}V_{p_j}(u_{x_j})|^{\frac{p_i+2}{3}}
    \mathcal{U}_{j}(h)^{\frac{(2-p_j)(p_i+2)}{12}} 
    \dx \right)^{\frac{3}{p_i+2}} \notag \\
    &\leq &  c \left( \int_{B_{t'}} |\tau_{j,h}V_{p_j}(u_{x_j})|^{2}  \dx \right)^{\!\frac{1}{2}}
    \left( \int_{B_{t'}} \mathcal{U}_{j}(h)^{\frac{(2-p_j)(p_i+2)}{2(4-p_i)}} 
    \dx \right)^{\!\frac{4-p_i}{2(p_i+2)}}
\end{eqnarray}
where, in the last step, we used Hölder’s inequality with exponents $(\frac{6}{p_i+2}, \frac{6}{4-p_i})$.\\
To estimate the last term of \eqref{j2apri}
notice that, since 
\begin{equation}\label{Stima4I2}
 p_j<\frac{p_i+2}{3}< p_i<p_j+1<\frac{3}{2}p_j+1,
\end{equation}
it holds
$$ p_j <\frac{(2-p_j)(p_i+2)}{4-p_i} \leq p_j+2.$$ 
 Therefore, there exists $\theta \in (0,1)$ such that
\begin{equation}
   \frac{\theta}{p_j}+ \frac{1-\theta}{p_j+2} = \frac{4-p_i}{(2-p_j)(p_i+2)} ,
\end{equation}
and using  the interpolation inequality in the last term of \eqref{j2apri}, we get
\begin{align}\label{Jpjmagg}
    |\mathbf{J}| &\leq
   c \left( \int_{B_{t'}} |\tau_{j,h}V_{p_j}(u_{x_j})|^{2} \, \dx \right)^{\!\frac{1}{2}} \notag \\
    &\cdot\left( 
       \int_{B_{t'}}
        \mathcal{U}_j(h)^\frac{p_j+2}{2} 
        \, \dx  
    \right)^\frac{(1-\theta)(2-p_j)}{2(p_j+2)} \left( 
       \int_{B_{t'}}
       \mathcal{U}_j(h)^\frac{p_j}{2} 
        \, \dx  
    \right)^{\frac{\theta(2-p_j)}{2p_j}} .
\end{align}
Summing estimates \eqref{Jpjmin} and \eqref{Jpjmagg} we conclude that
\begin{align}\label{A+BdiJ}
   |\mathbf{J}| &\leq  c
    \left(\ \int_{B_{t'}} |\tau_{j,h}V_{p_j}(u_{x_j})|^{2} \, \dx\right)^{\frac12}\Bigg[\left( \int_{B_{t'}} \mathcal{U}_j(h)^{\frac{p_j}{2}} 
    \dx \right)^{\!\frac{2-p_j}{2p_j}} \notag\\
   & \qquad+\left( \int_{B_{t'}} \mathcal{U}_j(h)^{\frac{p_j}{2}} 
    \dx \right)^{\frac{\theta(2-p_j)}{2p_j}}\left( 
       \int_{B_{t'}}
        \mathcal{U}_j(h)^\frac{p_j+2}{2} 
        \, \dx  
    \right)^\frac{(1-\theta)(2-p_j)}{2(p_j+2)}\Bigg] \notag\\
    &\leq  c
    \left(\ \int_{B_{t'}} |\tau_{j,h}V_{p_j}(u_{x_j})|^{2} \, \dx\right)^{\frac12}\notag\\
    &\quad\cdot \Bigg[\left( \int_{B_{t'}} \mathcal{U}_j(h)^{\frac{p_j}{2}} 
    \dx \right)^{\!\frac{2-p_j}{2p_j}} +\left( 
       \int_{B_{t'}}
        \mathcal{U}_j(h)^\frac{p_j+2}{2} 
        \, \dx  
    \right)^\frac{(2-p_j)}{2(p_j+2)}\Bigg],
\end{align}
where, in the last line, we used Young's inequality with exponents $\left(\frac{1}{\theta},\frac{1}{1-\theta}\right)$.
{Inserting \eqref{A+BdiJ} in \eqref{I3''split} and
using Young's inequality twice, we have}
\begin{align*}
    I_{3}''&\leq \frac{c |h|}{t - s} 
    \sum_{i=1}^{n} 
    \left( \int_{B_t} (\mu^2 + |u_{x_i}|^2)^\frac{p_i+2}{2} \, \dx \right)^{\frac{p_i-1}{p_i+2}}\left(\ \int_{B_{t'}} |\tau_{j,h}V_{p_j}(u_{x_j})|^{2} \, \dx\right)^{\frac12} \notag\\
   &\qquad \cdot\left[\left( \int_{B_{t'}} \mathcal{U}_j(h)^{\frac{p_j}{2}} 
    \dx \right)^{\!\frac{2-p_j}{2p_j}}+\left( 
       \int_{B_{t'}}
        \mathcal{U}_j(h)^\frac{p_j+2}{2} 
        \, \dx  
    \right)^\frac{(2-p_j)}{2(p_j+2)}\right] \notag\\ 
    &\leq \sigma \int_{B_{t'}} |\tau_{j,h}V_{p_j}(u_{x_j})|^{2} \, \dx+\frac{c_\varepsilon |h|^2}{(t - s)^2} 
    \sum_{i=1}^{n} 
    \left( \int_{B_t} (\mu^2 + |u_{x_i}|^2)^\frac{p_i+2}{2} \, \dx \right)^{\frac{2(p_i-1)}{p_i+2}}  \notag\\
   &\qquad\cdot \left[ \left( \int_{B_{t'}} \mathcal{U}_j(h)^{\frac{p_j}{2}} 
    \dx \right)^{\!\frac{2-p_j}{p_j}} 
    +\left( 
       \int_{B_{t'}}
        \mathcal{U}_j(h)^\frac{p_j+2}{2} 
        \, \dx  
    \right)^\frac{2-p_j}{p_j+2} \, \right]  \notag\\
    &\leq \sigma \int_{B_{t'}} |\tau_{j,h}V_{p_j}(u_{x_j})|^{2} \, \dx+\beta |h|^2  \sum_{i=1}^{n} \int_{B_t} (\mu^2 + |u_{x_i}|^2)^\frac{p_i+2}{2} \, \dx\notag\\&\qquad+\sum_{i=1}^{n}\frac{c_{\varepsilon,\beta} |h|^2}{(t - s)^{\frac{2(p_i+2)}{4-p_i}}} \left( \int_{B_{t'}} \mathcal{U}_j(h)^{\frac{p_j}{2}} 
    \dx \right)^{\!\frac{(2-p_j)(p_i+2)}{p_j(4-p_i)}}\notag\\
    &\qquad+\sum_{i=1}^{n}\frac{c_{\varepsilon,\beta} |h|^2}{(t - s)^{\frac{2(p_i+2)}{4-p_i}}}\left( 
       \int_{B_{t'}}
        \mathcal{U}_j(h)^\frac{p_j+2}{2} 
        \, \dx  
    \right)^\frac{(2-p_j){(p_i+2)}}{(p_j+2){(4-p_i)}},
 \end{align*}
 where $\sigma,\beta>0$ are as before.
 Since $$\frac{(2-p_j)(p_i+2)}{(p_j+2)(4-p_i)}<1\,\,\Longleftrightarrow\,\, p_i<\frac{3}{2}p_j+1$$
 which holds by  \eqref{crescitapi}, we may use Young's inequality again to derive
 \begin{eqnarray}\label{estI3''}
    I_{3}''&\leq& \sigma \int_{B_{t'}} |\tau_{j,h}V_{p_j}(u_{x_j})|^{2} \, \dx+\beta |h|^2  \sum_{i=1}^{n} \int_{B_t} (\mu^2 + |u_{x_i}|^2)^\frac{p_i+2}{2} \, \dx\cr\cr&&\qquad+\beta |h|^2\int_{B_{t'}}
        \mathcal{U}_j(h)^\frac{p_j+2}{2} 
        \, \dx +\sum_{i=1}^{n}\frac{c_{\varepsilon,\beta} |h|^2}{(t - s)^{\frac{2(p_i+2)}{4-p_i}}} \left( \int_{B_{t'}} \mathcal{U}_j(h)^{\frac{p_j}{2}} 
    \dx \right)^{\!\frac{(2-p_j)(p_i+2)}{p_j(4-p_i)}}\cr\cr&&\qquad + \frac{c_{\varepsilon,\beta} |h|^2}{(t - s)^{2\gamma}} \cr\cr
        &\leq& \sigma \int_{B_{t'}} |\tau_{j,h}V_{p_j}(u_{x_j})|^{2} \, \dx+2\beta |h|^2  \sum_{i=1}^{n} \int_{B_{\rho'}} (\mu^2 + |u_{x_i}|^2)^\frac{p_i+2}{2} \, \dx\cr\cr&&\quad+\frac{c_{\varepsilon,\beta} |h|^2}{(t - s)^{2\alpha}}\sum_{i=1}^{n} \left( \int_{B_{t'}} \mathcal{U}_j(h)^{\frac{p_j}{2}} 
    \dx \right)^{\!\frac{(2-p_j)(p_i+2)}{p_j(4-p_i)}} + \frac{c_{\varepsilon,\beta} |h|^2}{(t - s)^{2\gamma}},
 \end{eqnarray}
 where we used Lemma \ref{ldiff} and $\alpha=\alpha(p_i)$ and $\gamma=\gamma(p_i)$.
Inserting estimates \eqref{I3'completoo} and \eqref{estI3''} in \eqref{I2secondo}, we conclude that
 \begin{align}\label{estI3}
    |I_{3}|&\leq \sigma \int_{B_{t'}} |\tau_{j,h}V_{p_j}(u_{x_j})|^{2} \, \dx+3\beta |h|^2   \sum_{i=1}^{n} \int_{B_{\rho'}} (\mu^2 + |u_{x_i}|^2)^\frac{p_i+2}{2} \, \dx \notag\\
   &\qquad +\frac{c_{\varepsilon,\beta} |h|^2}{(t - s)^{2\alpha}} \sum_{i=1}^{n}\left( \int_{B_{t'}} \mathcal{U}_j(h)^{\frac{p_j}{2}} 
    \dx \right)^{\!\frac{(2-p_j)(p_i+2)}{p_j(4-p_i)}} + \frac{c_{\beta} |h|^2}{(t - s)^{2\gamma}}  \notag\\
   &\qquad+\frac{c_{\beta} |h|^2}{(t - s)^{2\alpha'}}   \int_{B_{t'}} \mathcal{U}_j(h)^{\frac{p_j}{2}} 
    \dx \notag \\
    &\leq \sigma \int_{B_{t'}} |\tau_{j,h}V_{p_j}(u_{x_j})|^{2} \, \dx+3\beta |h|^2   \sum_{i=1}^{n} \int_{B_{\rho'}} (\mu^2 + |u_{x_i}|^2)^\frac{p_i+2}{2} \, \dx \notag\\
   &\qquad +\frac{c_{\varepsilon,\beta} |h|^2}{(t - s)^{2\alpha''}} \left( \int_{B_{t'}} \mathcal{U}_j(h)^{\frac{p_j}{2}} 
    \dx \right)^{\!\tilde{\alpha}} + \frac{c_{\beta} |h|^2}{(t - s)^{2\gamma}} 
\end{align}
where $\tilde{\alpha}=\tilde{\alpha}(p_i)$ and $\alpha''=\alpha''(p_i)$ .\\
Inserting \eqref{I1}, \eqref{I3finaleapriori} and \eqref{estI3} in \eqref{SommaInt}, we get
\begin{eqnarray*}
 \ell\, \textbf{\textbf{RHS}} &\le&  \varepsilon \textbf{\textbf{RHS}} + 4\beta |h|^2  \sum_{i=1}^{n}   \int_{B_{{t'}}} \mathcal{U}_i(h)^{\frac{p_i+2}{2}} \dx+\frac{c_{\varepsilon,\beta} |h|^2}{(t - s)^{2\alpha''}} \left( \int_{B_{t'}} \mathcal{U}_j(h)^{\frac{p_j}{2}} 
    \dx \right)^{\!\tilde{\alpha}} \cr\cr
 &&\quad+ \sigma \int_{B_{t'}} |\tau_{j,h}V_{p_j}(u_{x_j})|^{2} \, \dx \frac{c_\beta |h|^2} {(t - s)^{2\tilde{\gamma}}}\left(\int_{B_{R}}(1+g(x))^r\dx\right)^{\frac{\tilde{\gamma}}{r}} ,
 \end{eqnarray*}
 where in the last integral we also included the contribution of $ \frac{c_{\varepsilon,\beta} |h|^2}{(t - s)^{2\gamma}}$.\\
Choosing $\varepsilon = \frac{\ell}{2}$, we can reabsorb the first integral in the right-hand side  by the left-hand side thus getting
\begin{eqnarray}\label{Formulacompleta}
 \frac{\ell}{2}\, \textbf{\textbf{RHS}} &\le&   4\beta |h|^2  \sum_{i=1}^{n}   \int_{B_{{t'}}} \mathcal{U}_i(h)^{\frac{p_i+2}{2}} \dx+\frac{c_{\varepsilon,\beta} |h|^2}{(t - s)^{2\alpha''}} \left( \int_{B_{t'}} \mathcal{U}_j(h)^{\frac{p_j}{2}} 
    \dx \right)^{\!\tilde{\alpha}} \cr\cr
 &&\quad+ \sigma \int_{B_{t'}} |\tau_{j,h}V_{p_j}(u_{x_j})|^{2} \, \dx \frac{c_\beta |h|^2} {(t - s)^{2\tilde{\gamma}}}\left(\int_{B_{R}}(1+g(x))^r\dx\right)^{\frac{\tilde{\gamma}}{r}},
 \end{eqnarray}
where we also used Lemma \ref{ldiff}. Recalling the definition of $ \textbf{\textbf{RHS}}$, \eqref{Formulacompleta} can be written as follows
\begin{eqnarray}\label{RHScompleto}
 \frac{\ell}{2}\,\sum_{i=1}^n \int_{B_{s}} |\tau_{j,h}V_{p_i}(u_{x_i})|^{2} \, \dx &\le&  \sigma \sum_{i=1}^{n}\int_{B_{t'}} |\tau_{j,h}V_{p_i}(u_{x_i})|^{2} \, \dx \notag \\
  &&\quad+ 4\beta |h|^2  \sum_{i=1}^{n}   \int_{B_{{\rho'}}} (\mu^2+|u_{x_i}|^2)^{\frac{p_i+2}{2}} \dx \notag \\
 &&\quad +\frac{c_{\beta} |h|^2}{(t - s)^{2\alpha''}} \left( \int_{B_{\rho'}} (\mu^2+|u_{x_j}|^2)^{\frac{p_j}{2}} 
    \dx \right)^{\!\tilde{\alpha}}\notag\\
   &&\quad  +\frac{c_\beta |h|^2} {(t - s)^{2\tilde{\gamma}}}\left(\int_{B_{R}}(1+g(x))^r\dx\right)^{\frac{\tilde{\gamma}}{r}},
 \end{eqnarray}
 that holds true for every $t \in (s,t')$  with constants independent of $s,t,t'$. Hence,  choosing $t= s+\frac{t'-s}{2}$ and $\sigma= \frac{\ell}{4}$, \eqref{RHScompleto} becomes
 \begin{eqnarray}\label{RHScompleto2}
 \frac{\ell}{2}\,\sum_{i=1}^n\int_{B_{s}} |\tau_{j,h}V_{p_i}(u_{x_i})|^{2} \, \dx &\le&  \frac{\ell}{4} \sum_{i=1}^{n}\int_{B_{t'}} |\tau_{j,h}V_{p_i}(u_{x_i})|^{2} \, \dx \notag \\
  &&\quad+ 4\beta |h|^2  \sum_{i=1}^{n}   \int_{B_{{\rho'}}} (\mu^2+|u_{x_i}|^2)^{\frac{p_i+2}{2}} \dx \notag \\
 &&\quad +\frac{c_{\beta} |h|^2}{(t' - s)^{2\alpha''}} \left( \int_{B_{\rho'}} (\mu^2+|u_{x_j}|^2)^{\frac{p_j}{2}} 
    \dx \right)^{\!\tilde{\alpha}}\notag\\
   &&\quad  +\frac{c_\beta |h|^2} {(t' - s)^{2\tilde{\gamma}}}\left(\int_{B_{R}}(1+g(x))^r\dx\right)^{\frac{\tilde{\gamma}}{r}}.
 \end{eqnarray}
 Since previous estimate holds true for every $\varrho<s<t'<\rho'$, we may use the iteration Lemma with the function
 \begin{eqnarray*}
     \phi(s)= \sum_{i=1}^n  \int_{B_{s}} |\tau_{j,h}V_{p_i}(u_{x_i})|^{2} \, \dx,
 \end{eqnarray*}
  thus getting
\begin{equation}\label{2phiRhs}
    \phi \left(\varrho \right) \leq  \frac{\tilde{C}_1}{(\rho'-\varrho)^{2\alpha''}} + \frac{\tilde{C}_2}{(\rho'-\varrho)^{2\tilde{\gamma}}}+\tilde{C_3}.
    \end{equation}
   Recalling the definition of $\phi(t)$ and  by Lemma \ref{Lemmahzero}, \eqref{2phiRhs} yields that 
\begin{align}\label{sommauxij}
 &    \sum_{i=1}^{n}   \left( \int_{B_{\varrho}} |u_{x_i x_j}|^2\left( \mu^2+|u_{x_i}|^2 \right)^\frac{p_i-2}{2} \dx \right)\le  c\beta   \sum_{i=1}^{n}   \int_{B_{{\rho'}}} (\mu^2+|u_{x_i}|^2)^{\frac{p_i+2}{2}} \dx \notag \\
    &\qquad\qquad+\frac{c_{\beta} }{(\rho'-\varrho)^{2\alpha''}} \left( \int_{B_{{\rho'}}} (\mu^2+|u_{x_j}|^2)^{\frac{p_j}{2}} 
    \dx \right)^{\tilde{\alpha}} \notag \\
   &\qquad\qquad+  \frac{c_\beta } {(\rho'-\varrho)^{2\tilde{\gamma}}}\left(\int_{B_{R}}(1+g(x))^r\dx\right)^{\frac{\tilde{\gamma}}{r}},
\end{align}
Choosing $j=i$, \eqref{sommauxij}  implies 
\begin{align*}
 &       \int_{B_{\varrho}} |u_{x_i x_i}|^2\left( \mu^2+|u_{x_i}|^2 \right)^\frac{p_i-2}{2} \dx \le  c\beta   \sum_{i=1}^{n}   \int_{B_{{{\rho'}}}} (\mu^2+|u_{x_i}|^2)^{\frac{p_i+2}{2}} \dx \notag \\
    &\qquad\qquad+\frac{c_{\beta} }{(\rho'-\varrho)^{2\alpha''}} \left( \int_{B_{{\rho'}}} (\mu^2+|u_{x_i}|^2)^{\frac{p_i}{2}} 
    \dx \right)^{\tilde{\alpha}} \notag \\
   &\qquad\qquad+  \frac{c_\beta } {(\rho'-\varrho)^{2\tilde{\gamma}}}\left(\int_{B_{R}}(1+g(x))^r\dx\right)^{\frac{\tilde{\gamma}}{r}},
\end{align*}
for every $i=1, \dots, n$.
Hence, summing over $i=1,\dots, n$, we infer
\begin{align}\label{sommasuj}
     \sum_{i=1}^n   & \int_{B_{{\varrho}} } |u_{x_i x_i}|^2\left( \mu^2+|u_{x_i}|^2 \right)^{\frac{p_i-2}{2}} \dx 
 \leq cn\beta   \sum_{i=1}^{n}   \int_{B_{{\rho'}}} (\mu^2+|u_{x_i}|^2)^{\frac{p_i+2}{2}} \dx \notag \\
   &\qquad\qquad+\frac{c_{\beta} }{(\rho'-\varrho)^{2\alpha''}} \left( \int_{B_{{\rho'}}} (\mu^2+|u_{x_i}|^2)^{\frac{p_i}{2}} 
    \dx \right)^{\tilde{\alpha}} \notag \\
   &\qquad\qquad+  \frac{c_\beta } {(\rho'-\varrho)^{2\tilde{\gamma}}}\left(\int_{B_{R}}(1+g(x))^r\dx\right)^{\frac{\tilde{\gamma}}{r}}. 
 \end{align}
Now, recall that 
$u_{x_i} \in L^{p_i+2}_{\mathrm{loc}}(\Omega )$
 for every $i=1,\dots, n$ and the estimate
 \begin{align}\label{TesiProprof}
     \sum_{i=1}^{n} \int_{B_{{\rho'}}} \left( \mu^2 +|u_{x_i}|^2 \right)^\frac{p_i+2}{2} \dx &\leq c \sum_{i=1}^{n} \, \int_{B_{R'}} \left( \mu^2 + |u_{x_i}|^2 \right)^\frac{p_i-2}{2} |u_{x_ix_i}|^2  \dx\notag \\
     &\quad+  \frac{c}{(R'-{\rho'})^{\delta}} \sum_{i=1}^{n} \int_{B_{R'}} \left( \mu^2 + |u_{x_i}|^2 \right)^\frac{p_i}{2} \dx 
   \end{align}
   holds with constants  $c=c(n, p_i,||u||^2_{\infty}  )$ and $\delta=\delta(p_i).$
    Using \eqref{TesiProprof} to estimate the first term in the right hand side of \eqref{sommasuj}, we find 
   \begin{align}\label{uxipi+2}
     \sum_{i=1}^n &   \int_{B_{{\varrho}}}  |u_{x_i x_i}|^2\left( \mu^2+|u_{x_i}|^2 \right)^{\frac{p_i-2}{2}} \dx  
 \leq   c\beta  \sum_{i=1}^{n} \int_{B_{R'}}   |u_{x_i x_i}|^2\left( \mu^2+|u_{x_i}|^2 \right)^{\frac{p_i-2}{2}} \dx \notag \\
 & \qquad +\frac{c_{\beta} }{({\rho'}-{\varrho})^{2\alpha''}}\sum_{i=1}^n \left( \int_{B_{{R}}} (\mu^2+|u_{x_i}|^2)^{\frac{p_i}{2}}\dx \right)^{\tilde{\alpha}}+  \frac{c}{(R'-{\rho'})^{\delta}} \sum_{i=1}^{n} \int_{B_{{R}}} \left( \mu^2 + |u_{x_i}|^2 \right)^\frac{p_i}{2} \dx\notag\\
 &\qquad+\frac{c_\beta } {({\rho'}-{\varrho})^{2\tilde{\gamma}}}\left(\int_{B_{R}}(1+g(x))^r\dx\right)^{\frac{\tilde{\gamma}}{r}} .
 \end{align}
 Since previous estimate holds for all radii $\rho< \varrho <\rho'<R'<R$ with constants independent of the radii, in particular it applies  choosing
   $$\rho'= \varrho +  \frac{4(R'-\varrho)}{5}, $$
  so that $\rho'-\varrho=\frac{4(R'- \varrho)}{5}$ and $R'-\rho'=\frac{(R'- \varrho)}{5}$.   Hence, 
choosing $\beta >0$ such that $ c\beta =\frac12$ and setting  $$ \varphi(\rho')= \sum_{i=1}^n   \left( \int_{B_{\rho'}} |u_{x_ix_i}|^2 \left( \mu^2 + |u_{x_i}|^2 \right)^\frac{p_i-2}{2}  \dx\right), $$
inequality \eqref{uxipi+2} can be written as follows
    \begin{equation*}\label{phirho}
\varphi(\varrho) \leq  \frac12   \varphi(R')+  \frac{{C_{1}}}{(R'-\varrho)^{2\alpha''}} + \frac{{C_{2}}}{(R'-\varrho)^{\delta}}+  \frac{C_3}{(R'-\varrho)^{2\tilde{\gamma}}}.
\end{equation*}
 that holds  for every $\rho<\varrho<\rho'<R'< R$.  Lemma \ref{lm3} implies
\begin{equation*}\label{phiR}
    \phi \left(\rho \right) \leq  \frac{\tilde{C}_1}{(R-\rho)^{2\alpha''}} + \frac{\tilde{C}_2}{(R-\rho)^{\delta}}+\frac{\tilde{C}_3}{(R-\rho)^{2\tilde{\gamma}}}.
    \end{equation*}
   i.e.
    \begin{align*}
        \sum_{i=1}^{n}   \left( \int_{B_{\rho}} \left( \mu^2 +|u_{x_i}|^2 \right)^\frac{p_i-2}{2}  |u_{x_i x_i}|^2 \dx \right) & \leq c  \left( \sum_{i=1}^{n} \Vert u_{x_i} \Vert_{L^{p_i}(B_{R})} + \Vert g \Vert_{L^r (B_{R})} \right)^\sigma ,
\end{align*}
where $c=c(p_i,n,\ell, L, \rho,R, ||u||_{\infty})$ and $\sigma=\sigma(n,p_i)$.
Putting previous estimate  in \eqref{TesiProprof}, we also derive
\begin{align}\label{pi+2}
 \sum_{i=1}^n \int_{B_{\rho}} \left( \mu^2 + |u_{x_i}|^2 \right)^\frac{p_i+2}{2} \dx & \leq c \left( \sum_{i=1}^{n} \Vert u_{x_i} \Vert_{L^{p_i}(B_{R})} + \Vert g \Vert_{L^r (B_{R})} \right)^\sigma  , 
\end{align}
for every pair of concentric balls $B_\rho\subset B_R\Subset \Omega$ and where $c = c(p_i,n,\ell, L, \rho,R, ||u||_{\infty})$ and $\sigma= \sigma (n,p_i)$ are positive constants.\\
Now, using \eqref{pi+2} in  the right hand side of \eqref{sommauxij}, we get 
\begin{align}\label{Stimacompattauxixjpi+2}
        \sum_{i=1}^{n}   \left( \int_{B_{\rho}} \left( \mu^2 +|u_{x_i}|^2 \right)^\frac{p_i-2}{2}|u_{x_ix_j}|^2 \dx \right) & \leq c  \left( \sum_{i=1}^{n} \Vert u_{x_i} \Vert_{L^{p_i}(B_{R})} + \Vert g \Vert_{L^r (B_{R})} \right)^\sigma,
\end{align}
for every $i=1, \dots, n$ and  for every $j=1, \dots, n$
and this concludes the proof.
\end{proof}

\section{Approximation}\label{Appro}
In this section, we complete the proof of Theorem \ref{thmPrincipale} by using the a priori estimate of Theorem \ref{AppThm} together with a classical approximation argument. The computations follow those in \cite{Russo} and are included here for completeness.\\
Fix a non-negative smooth kernel $\phi \in \mathcal{C}^\infty_0(B_1(0))$ such that $\int_{B_1(0)} \phi =1$ and consider the corresponding family of mollifiers $(\phi_{\delta})_{\delta >0}$. We consider 
$$f^\delta (x,\xi)=\int_{B_1(0)} f(x+\delta y,\xi)\phi(y)\,dy,$$
 and, for a fixed ball $B_R\Subset\Omega$,  introduce the corresponding  integral functionals $$\mathcal{F}^{{\delta}}(u,B_R):=\int_{_{B_R} }f^{{\delta}}(x, \xi) \dx$$ and
$$\mathcal{F}^{{\delta}}_{\varepsilon}(u,B_R):= \, \int_{B_R} \left( f^{{\delta}}(x,\xi)+ \varepsilon(1+ |\xi|^2)^{\frac{p_n}{2}} \right) \dx.$$
It is easy to check that $f^{{\delta}}(x, \xi)$ satisfies assumptions \eqref{crescitapi}, \eqref{A0}\mbox{---}\eqref{A3'}. Moreover  it satisfies \eqref{A'4}, with $K = \sup_{B_R} |g_{{\delta}}(x)| $ where $g_{{\delta}} =g\star\phi_{{\delta}}$ is the mollification of $g$ with the sequence of mollifiers $\phi_{{\delta}}$.
\\
We now consider the following variational problems
\begin{equation}\label{Pfj}
    {\min} \left\{  \int_{B_R} \left( f^{{\delta}}(x,Dv)+ \varepsilon(1+ |Dv|^2)^{\frac{p_n}{2}} \right) \dx \ : \ v \in W^{1,p_n}_0(B_R)+u^\eta \right\},
\end{equation}
where $$u^\eta = u \star \varsigma_\eta$$ is the mollification of the local minimizer $u$ of \eqref{functional}, with a sequence of mollifier that we denoted by $\varsigma_n$.\\
It is well known that, by the direct methods of the calculus of variations, there exists a unique solution $u^\eta_{{\delta},\varepsilon} \in  W^{1,p_n}_0(B_R)+u^\eta $ of  problem \eqref{Pfj}. Since the integrand $f^{{\delta}}(x,Dv)+ \varepsilon(1+ |Dv|^2)^{\frac{p_n}{2}}$ satisfies the assumptions of  Lemma \ref{LemmaInduzione}, we have that $$V_{p_i}((u^\eta_{{\delta},\varepsilon})_{x_i}) \in W_{\mathrm{loc}}^{1,2}(B_R),$$
for every $\eta, \delta$ and $\varepsilon>0. $ Hence we are legitimate to use estimate \eqref{uxistima} of Theorem \ref{AppThm}, to obtain 
\begin{align*}
    \sum_{i=1}^{n}  &\left(  \int_{B_{\rho}} \left( \mu^2 +|(u^\eta_{{\delta},\varepsilon})_{x_i}|^2 \right)^\frac{p_i +2 }{2} \dx \right) \\
      &\leq c \left( { \sum_{i=1}^{n} \int_{B_R}\left( \mu^2 +|(u^\eta_{{\delta},\varepsilon})_{x_i}|^2 \right)^\frac{p_i }{2} \, \dx }+ \Vert g_{{\delta}} \Vert_{L^r (B_{R})} \right)^\sigma \\
    & \le c  \left(  \int_{B_R} f^{{\delta}}(x,D u^\eta_{{\delta},\varepsilon}) \dx+ \Vert g_{{\delta}} \Vert_{L^r (B_{R})} \right)^\sigma \\
    & \leq c  \left(  \int_{B_R} \left(f^{{\delta}}(x,D u^\eta_{{\delta},\varepsilon}) + \varepsilon(1+|Du^\eta_{{\delta},\varepsilon}|^2)^{\frac{p_n}{2}} \right) \dx+ \Vert g_{{\delta}} \Vert_{L^r (B_{R})} \right)^\sigma  ,
\end{align*}
where  $c$ is a positive constant that is independent on $\varepsilon, {\delta}$ and $\eta$ and {where  we used assumption \eqref{crescitapi}}. Using  the minimality of $u^\eta_{{\delta},\varepsilon}$ in the right-hand side of previous estimate, we infer that
\begin{align}\label{stimaukeni}
    &\sum_{i=1}^{n} \left(  \int_{B_{\rho}}\left( \mu^2 +| (u^\eta_{{\delta},\varepsilon})_{x_i}|^2 \right)^\frac{p_i +2 }{2} \dx \right) \notag \\
     & \leq c  \left(  \int_{B_R} \left(f^{{\delta}}(x,D u^\eta) + \varepsilon(1+|Du^\eta|^2)^{\frac{p_n}{2}} \right) \dx+ \Vert g_{{\delta}} \Vert_{L^r (B_{R})} \right)^\sigma,
\end{align}
where $c$ is a positive constant independent  of $\varepsilon,\eta,{\delta}$.\\
    Since $u^\eta \in W^{1,p_n}(B_R)$, the right-hand side of the previous estimate is finite and therefore the left-hand side  is uniformly bounded  with respect to $\varepsilon$, for $\varepsilon \in (0,1)$. Therefore, there exists $v^\eta_k$ such that, up to a not relabeled subsequence,
    $$u^\eta_{{\delta},\varepsilon} \rightharpoonup v^\eta_{{\delta}} \text{  weakly in  } W^{1, \textbf{p}+2}(B_R) \, \text{ as  } \varepsilon \to 0.$$
By the weak lower semicontinuity of the norm, we get
\begin{align}\label{stimaconv1}
    \sum_{i=1}^{n}& \left(  \int_{B_{\rho}} {\left( \mu^2 +|(v^\eta_{{\delta}})_{x_i}|^2 \right)^\frac{p_i +2 }{2} } \dx \right) \notag \\& \leq \liminf_{\varepsilon}\sum_{i=1}^{n} \int_{B_{\rho}}\left( { \left( \mu^2 +|(u^\eta_{{\delta},\varepsilon})_{x_i}|^2 \right)^\frac{p_i +2 }{2} \dx }\right) \notag \\
    &\le c  \left(  \int_{B_R}f^{{\delta}}(x,D u^\eta)  \dx+ \Vert g_{{\delta}} \Vert_{L^r (B_{R})} \right)^\sigma ,
\end{align}
where we used \eqref{stimaukeni} and that, since $u^\eta \in W^{1, p_n}(\Omega)$, it holds 
\begin{equation}\label{fkappalim}
\liminf_{\varepsilon \to 0} \, \varepsilon \int_{B_R}(1+|Du^\eta|^2)^{\frac{p_n}{2}} \dx =0. 
\end{equation}
 {Since $f(x,Du^\eta)\in L^1(B_R)$ and $g\in L^r(B_R)$, by the properties of the mollifications, it hold
\begin{equation}\label{fktoF}
   \lim_{\delta\to 0^+}\int_{B_R} f^{\delta}(x, Du^\eta) \dx = \int_{B_R}f(x,Du^\eta) \dx.
    \end{equation}
    and
    \begin{equation}\label{gktog}g_{\delta} \to g\quad \text{in}\quad L^r(B_R).\end{equation}}
 So, by \eqref{fktoF} and \eqref{gktog}, the right-hand side of \eqref{stimaconv1} is uniformly bounded with respect to {$\delta$}. Therefore, there exists $v^\eta$ such that, up to a subsequence,
    $$v^\eta_{{\delta}} \rightharpoonup v^\eta \text{  weakly in  } W^{1, \textbf{p}+2}(B_R) \, \text{ as  } {\delta} \to 0.$$
Again by the weak lower semicontinuity of the norm, \eqref{stimaconv1} and \eqref{fktoF}, we get
\begin{align}\label{stimaconv2}
    &\sum_{i=1}^{n} \left(  \int_{B_{\rho}} {\left( \mu^2 +| (v^\eta)_{x_i}|^2 \right)^\frac{p_i+2}{2} } \dx \right)\notag \\ & \leq \liminf_{{\delta}\to 0}\sum_{i=1}^{n} \left(  \int_{B_{\rho}} {\left( \mu^2 +| (v^\eta_{{\delta}})_{x_i}|^2 \right)^\frac{p_i+2}{2} } \dx \right) \notag \\
    &\leq c  \left(  \int_{B_R}f(x,D u^\eta)  \dx+ \Vert g \Vert_{L^r (B_{R})} \right)^\sigma .
\end{align}
Since, by the properties of mollification, it holds 
$$u^\eta_{x_i} \to u_{x_i}  \text{ strongly in }L^{p_i},$$
 by assumption \eqref{A3'} we have
\begin{equation}\label{IntfDuni}
   \lim_{\eta \to 0^+} \int_{B_R}f(x, Du^\eta)\dx =  \int_{B_R}f(x,Du) \dx.
    \end{equation}
 By virtue of \eqref{IntfDuni},  the right-hand side of \eqref{stimaconv2} can be uniformly bounded w.r.t. $\eta$. Therefore there exists $v$ such that, up to a subsequence,
    $$v^\eta \rightharpoonup v \text{  weakly in  } W^{1, \textbf{p}+2}(B_R) \, \text{ as  } \eta \to 0.$$
Once again, by the weak lower semicontinuity of the norm, we derive
\begin{align}\label{stimaconv3}
  { \sum_{i=1}^{n} \left(  \int_{B_{\rho}}\left( \mu^2+|v_{x_i}|^2 \right)^\frac{p_i+2}{2} \dx \right)} & \leq {\liminf_{\eta}\sum_{i=1}^{n} \left(  \int_{B_{\rho}}\left( \mu^2+|(v^\eta)_{x_i}|^2 \right)^\frac{p_i+2}{2} \dx \right)} \notag \\
    & \leq \lim_{\eta}  c  \left(  \int_{B_R}f(x,D u^\eta)  \dx+ \Vert g \Vert_{L^r (B_{R})} \right)^\sigma  \notag \\
    &= c  \left( \int_{B_R}f(x,D u)  \dx+ \Vert g \Vert_{L^r (B_{R})} \right)^\sigma. 
\end{align}
In order to conclude the proof, it suffices to show that $u=v$ a.e. in $B_R$.\\
By the weak lower semicontinuity of the functionals  $\mathcal{F}$ and $\mathcal{F}^{\delta}$ in $W^{1, \textbf{p}}(B_R)$ , the weak convergence of $v^\eta \rightharpoonup v$ in $W^{1, \textbf{p}+2}(B_R)$, \eqref{fktoF} and $v^\eta_\delta \rightharpoonup v^\eta$ in $W^{1, \textbf{p}+2}(B_R)$, we get
\begin{align}
    \int_{B_R}f(x, Dv) \dx & \leq \liminf_{\eta} \int_{B_R} f(x, Dv^\eta) \dx \leq \liminf_{\eta} \liminf_{{\delta}} \int_{B_R} f^{{\delta}}(x, Dv^\eta) \dx \notag \\   
    & \leq \liminf_{\eta}\liminf_{{\delta}} \int_{B_R} f^{{\delta}}(x, Dv^\eta_{{\delta}}) \dx \notag \\
     & \leq \liminf_{\eta}\liminf_{{\delta}} \int_{B_R} \left( f^{{\delta}}(x, Dv^\eta_{{\delta}}) )+ \varepsilon(1+ |Du^\eta_{\delta}|^2)^{\frac{p_n}{2}}   \right)\dx, \notag 
    \end{align}
   where the last inequality is trivial, by the non-negativity of the last term.\\
    Thanks to the weak convergence of $u^\eta_{{\delta}, \varepsilon} \rightharpoonup v^\eta_{{\delta}}$ in $W^{1, \textbf{p}+2}(B_R)$ and again the weak lower semicontinuity of functional $\mathcal{F}^{{\delta}}_\varepsilon$, we have 
   \begin{align}\label{pallino}
    \int_{B_R}f(x, Dv) \dx & \leq \liminf_{\eta}\liminf_{{\delta}}\liminf_{\varepsilon} \int_{B_R} \left( f^{{\delta}}(x, Du^\eta_{\delta, \varepsilon})+ \varepsilon(1+ |Du^\eta_{{\delta}, \varepsilon}|^2)^{\frac{p_n}{2}}  \right) \dx \notag \\
     & \leq \liminf_{\eta}\liminf_{{\delta}}\liminf_{\varepsilon} \int_{B_R} \left( f^{{\delta}}(x, Du^\eta)+ \varepsilon(1+ |Du^\eta|^2)^{\frac{p_n}{2}} \right)\dx \notag \\
      & = \liminf_{\eta}\liminf_{{\delta}} \int_{B_R} f^{{\delta}}(x, Du^\eta) \dx , 
      \end{align} 
      where we used the minimality of $u^\eta_{{\delta}, \varepsilon}$ and \eqref{fkappalim}.
      By \eqref{pallino},  \eqref{fktoF} and \eqref{IntfDuni}, we derive
      \begin{align}
    \int_{B_R}f(x, Dv) \dx  \leq \liminf_{\eta} \int_{B_R} f(x, Du^\eta) \dx  \leq  \int_{B_R} f(x, Du) \dx,
\end{align}
and then, by minimality of  $u$,  we infer
\begin{equation*}
     \int_{B_R} f(x, Dv) \dx =\int_{B_R} f(x, Du) \dx.
\end{equation*}
By the strict convexity of $\xi \to f(x, \xi)$, we deduce that $u=v$. Then $u \in W^{1, \textbf{p}+2}$ and so we can argue as in the proof of Theorem \ref{AppThm}, thus obtaining 
\begin{align}
       \sum_{i=1}^{n}  { \left( \int_{B_{\rho}} |u_{x_i x_j}|^2 \left( \mu^2+ |u_{x_i}|^2 \right)^\frac{p_i-2}{2} \dx \right)} & \leq c  \left( \sum_{i=1}^{n} \Vert u_{x_i} \Vert_{L^{p_i}(B_R)} + \Vert g \Vert_{L^r (B_{R})} \right)^\sigma ,
\end{align}
and this conclude the proof.

\bigskip
\bigskip

\par\noindent {\bf Data availability statement.} Data sharing not applicable to this article as no datasets were generated or analysed during the current study.

\smallskip
\par\noindent
{\bf Funding}. This research was partly funded by:
\\ (i) GNAMPA   of the Italian INdAM - National Institute of High Mathematics (grant number not available)  (A.Passarelli di Napoli, S. Russo);
\\ (ii)  GNAMPA Project 2026, grant number  E53C25002010001, ``Esistenza e regolarità per soluzioni di equazioni ellittiche e paraboliche anisotrope'' ( A.Passarelli di Napoli, S. Russo);
\\(iii)   Centro Nazionale per la Mobilità Sostenibile (CN00000023) - Spoke 10 Logistica Merci, grant number E63C22000930007,  funded by PNRR (A.Passarelli di Napoli, S. Russo);
\\(iv)Spanish State Research Agency, through the Severo Ochoa and María de Maeztu Program for Centers and Units of Excellence in R$\&$D, grant number (CEX2020001084M), (A.Clop);
\\(v) Government of Spain and the  Catalan Research Agency (Government of Catalonia)-Project PID2024-156429, (A.Clop).

\bigskip
\par\noindent
{\bf
Acknowledgment.} Part of this project was carried out while S. Russo was a visiting researcher at the University of Barcelona (UB). She warmly thanks the institution for its hospitality and Professor Clop for his support. 

\bigskip
\par\noindent
{\bf Conflict of Interest}. The authors declare that they have no conflict of interest.

\Addresses
\end{document}